\documentclass[12pt]{amsart}
\usepackage[mathcal]{eucal}
\usepackage{amssymb, graphicx, labelfig}
\usepackage[all]{xy}

\newtheorem{thm}{Theorem}
\newtheorem{prop}[thm]{Proposition}
\newtheorem{lem}[thm]{Lemma}

\newtheorem{cor}[thm]{Corollary}

\theoremstyle{remark}
\newtheorem{rem}[thm]{Remark}

\theoremstyle{definition}

\newcommand{\C}{\mathbb C}
\newcommand{\R}{\mathbb R}
\newcommand{\Z}{\mathbb Z}

\newcommand{\Id}{\mathrm{Id}}
\newcommand{\End}{\mathrm{End}}

\newcommand{\SSS}{\mathcal S^A(S)}
\newcommand{\SSSS}{\mathcal S^{\epsilon}(S)}
\newcommand{\SSSSS}{\mathcal S^A_{\mathrm s}}
\newcommand{\TT}{\mathcal T^\omega(\lambda)}
\newcommand{\TTT}{\mathcal T^{\iota}(\lambda)}
\newcommand{\RR}{\mathcal R_{\SL(\C)}(S)}
\newcommand{\RS}{\mathcal R_{\PSL(\C)}^{\mathrm{Spin}}(S)}
\newcommand{\Spin}{\mathrm{Spin}(S)}
\newcommand{\SL}{\mathrm{SL}_2}
\newcommand{\PSL}{\mathrm{PSL}_2}
\newcommand{\Tr}{\mathrm{Tr}}
\newcommand{\E}{\mathrm{e}}
\newcommand{\I}{\mathrm{i}}
\newcommand{\QBinom}[3] {\begin{pmatrix}{#1}\\{#2}\end{pmatrix}_{\kern -4pt{#3}}}
\newcommand{\QInt}[2]{\left(#1\right)_{\kern -1 pt{#2}}}
\newcommand{\QQInt}[2]{\left[#1\right]_{\kern 0 pt{#2}}}

\newcommand{\db}{/\kern -3pt/}

\renewcommand{\leq}{\leqslant}
\renewcommand{\geq}{\geqslant}
\renewcommand{\phi}{\varphi}
\renewcommand{\epsilon}{\varepsilon}

\title
[Representations of the skein algebra I]
{Representations of \\
the Kauffman bracket skein algebra I:\\ invariants and miraculous cancellations}

\author{Francis Bonahon}
\address {Department
of Mathematics,  University of
Southern California, Los Angeles
CA~90089-2532, U.S.A.}
\email{fbonahon@math.usc.edu}

\author{Helen Wong}
\address{Department
of Mathematics, Carleton College, Northfield MN 55057, U.S.A.}
\email{hwong@carleton.edu}

\thanks{This research was partially supported by grants DMS-0604866, DMS-1105402 and DMS-1105692  from the National Science Foundation, and by a mentoring grant from the Association for Women in Mathematics.}
\date{\today}

\begin{document}

\begin{abstract}
We study finite-dimensional  representations of the  Kauffman bracket skein algebra of a surface $S$. In particular, we construct invariants of such irreducible representations when  the underlying parameter $q=\E^{2\pi \I \hbar}$ is a root of unity. The main  one of these invariants is a point in the  character variety consisting of group homomorphisms from the fundamental group $\pi_1(S)$ to  $\SL(\C)$, or in a twisted version of this character variety. The proof relies on certain miraculous cancellations that occur for the quantum trace homomorphism constructed by the authors. These miraculous cancellations  also play a fundamental role in subsequent work of the authors, where  novel examples of representations of the skein algebra are constructed. 
\end{abstract}
\maketitle

For an oriented surface $S$ of finite topological type and for a Lie group $G$, many areas of mathematics involve the \emph{character variety}
$$
\mathcal R_G(S) = \{ \text{group homomorphisms } \pi_1(S) \to G \} \db G,
$$
where $G$ acts on homomorphisms by conjugation. For $G=\SL(\C)$, Turaev \cite{Tur} showed that the corresponding character variety $\RR$ can be quantized by the Kauffman bracket skein algebra of the surface; see also \cite{BFK1, BFK2, PrzS}. In fact, if one follows the physical tradition that a quantization of a space $X$ replaces the commutative algebra of functions on $X$ by a non-commutative algebra of operators on a Hilbert space, the points of an actual quantization of the character variety $\RR$ should be \emph{representations} of the Kauffman bracket skein algebra. 

This article studies finite-dimensional representations of the skein algebra of a surface. The \emph{Kauffman bracket skein algebra} $\SSS$  depends on a parameter $A=\E^{\pi\I \hbar}\in \C-\{0\}$, and is defined as follows. One first considers the vector space freely generated by  all isotopy classes of framed links in the thickened surface $S \times [0,1]$, and then one takes the quotient of this space by two relations: the main one is the  \emph{skein relation} that
$$
[K_1] = A^{-1} [K_0] + A [K_\infty]
$$
whenever the three links $K_1$, $K_0$ and $K_\infty\subset S\times [0,1]$ differ only in a little ball where they are as represented on Figure~\ref{fig:SkeinRelation}; the second relation states that $[O] = -(A^2+A^{-2})[\varnothing]$ for the trivial framed knot $O$ and the empty link $\varnothing$. The algebra multiplication is defined by  superposition of skeins. See \S \ref{sect:SkeinAlgebra} for details.

\begin{figure}[htbp]

\SetLabels
( .5 * -.4 ) $K_0$ \\
( .1 * -.4 )  $K_1$\\
(  .9*  -.4) $K_\infty$ \\
\endSetLabels
\centerline{\AffixLabels{\includegraphics{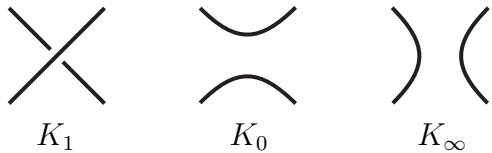}}}
\vskip 15pt
\caption{A Kauffman triple}
\label{fig:SkeinRelation}
\end{figure}

Our goal is to study representations of the skein algebra, namely algebra homomorphisms $\rho \colon \SSS \to \End(V)$ where $V$ is a finite-dimensional vector space over $\C$. See \cite{BonWon2} for an interpretation of such representations as generalizations of the Kauffman bracket invariant of framed links in $\R^3$. When $A$ is a root of unity, a well-known example of a finite-dimensional representation of the skein algebra  $\SSS$  arises from the Witten-Reshetikhin-Turaev topological quantum field theory associated to the fundamental representation of the quantum group $\mathrm U_q (\mathfrak{sl}_2)$ \cite{Witten, ReshTur, BHMV, TuraevBook}.

To analyze such representations, we need to find  a way to distinguish them.  We thus introduce invariants for these representations.   

We focus attention on the case where $A^2$ is a root of unity, and more precisely    where $A^2$ is a primitive $N$--root of unity with $N$ odd. While the restriction to roots of unity is natural, the parity condition for $N$  is forced on us by the mathematics. 

The main invariant comes from Chebyshev polynomials of the first kind. The \emph{$n$--th normalized Chebyshev polynomial of the first kind} is the polynomial $T_n(x)$  determined by the trigonometric identity that $2\cos n\theta = T_n(2\cos \theta)$.
More precisely, if $A^2$ is a primitive $N$-root of unity with $N$  odd, we consider the Chebyshev polynomial $T_N$.  We separate the discussion into two cases, depending on whether $A^N = +1$ or $A^N=-1$.  

\begin{thm}
\label{thm:ShadowIntroA=-1}
Suppose that $A^2$ is a primitive $N$--root of unity with $N$ odd, and that $A^N=-1$. Then, for every finite-dimensional irreducible representation $\rho \colon \SSS \to \End(V)$ of the skein algebra, there exists a unique character $r_\rho\in\RR$ such that
$$
T_N \bigl( \rho ([K]) \bigr) =- \bigl( \Tr\, r_\rho(K) \bigr) \Id_V
$$
for every framed knot $K\subset S \times [0,1]$ whose projection to $S$ has no crossing and whose framing is vertical. (Here, $\Tr\,r_\rho(K) \in \C$ denotes the trace of $r_\rho(K) \in \SL(\C)$.)
\end{thm}

In particular, interpreting the elements of $\End(V)$ as matrices, many ``miraculous cancellations'' occur in the entries of $T_N \bigl( \rho ([K]) \bigr)$ when one evaluates the Chebyshev polynomial $T_N$ over the endomorphism $\rho\bigl ([K]\bigr)\in \End(V)$. A more general version of Theorem~\ref{thm:ShadowIntroA=-1} and of these miraculous cancellations, valid for all framed links in $S\times[0,1]$, is provided by Theorem~\ref{thm:Shadow} in \S\ref{sect:Shadow} below. 

In general, researchers working on the Kauffman bracket skein algebra (or on the representation theory of the quantum group $\mathrm U_q(\mathfrak{sl}_2)$) are more familiar with the normalized  Chebyshev polynomials of the second kind, defined as the polynomials  $S_n(x)$ such that $\sin (n+1)\theta = \sin \theta\, S_n(2\cos \theta)$. The occurrence of the other Chebyshev polynomials $T_n(x)$ is here somewhat surprising\footnote{However, see \cite{FroGel} for an earlier occurrence, as well as \cite{HavPos} for a related but different context.}.

There is a companion statement to Theorem~\ref{thm:ShadowIntroA=-1} when $A^N=+1$, except that it now involves a twisted product $\RS= \mathcal R_{\PSL(\C)}^0 (S)\widetilde\times \Spin$ of a component of the character variety $ \mathcal R_{\PSL(\C)} (S)$ with the set $\Spin$ of isotopy classes of spin structures on $S$. This twisted character variety $\RS$ is more natural than would appear at first glance; for instance, the monodromy of a hyperbolic metric on $S\times (0,1)$ determines an element of $\RS$. See \S \ref{sect:Shadow} for details.

The definition is designed so that, for every twisted character $r \in \RS$ and for every framed knot $K\subset S \times [0,1]$, there is a well-defined trace $\Tr\,r(K)\in \C$. 

\begin{thm}
\label{thm:ShadowIntroA=+1}
Suppose that $A^2$ is a primitive $N$--root of unity with $N$ odd, and that $A^N=+1$.  Then, for every finite-dimensional irreducible representation $\rho \colon \SSS \to \End(V)$ of the skein algebra, there exists a unique twisted character $r_\rho\in\RS$ such that
$$
T_N \bigl( \rho ([K]) \bigr) = -\bigl( \Tr\, r_\rho(K) \bigr) \Id_V
$$
for every framed knot $K\subset S \times [0,1]$ whose projection to $S$ has no crossing and with vertical framing. 
\end{thm}

Again, a more general version of Theorem~\ref{thm:ShadowIntroA=+1}  for all framed links in $S\times[0,1]$ is provided by Theorem~\ref{thm:Shadow} in \S\ref{sect:Shadow}.

For a finite-dimensional irreducible representation $\rho \colon \SSS \to \End(V)$, Theorems~\ref{thm:ShadowIntroA=-1} and \ref{thm:ShadowIntroA=+1} both associate to this quantum object a point in the character variety $\RR$ or $\RS$, which in particular is a classical (= non-quantum) geometric  object. We call $r_\rho\in \RR$ or $\RS$ the \emph{classical shadow} of the representation $\rho \colon \SSS \to \End(V)$.

For instance, when $A^N=-1$, we can consider the  finite-dimensional representation $\rho_{\mathrm{WRT}} \colon \SSS \to \End(V)$ provided by  the $\mathrm{SO}(3)$ version of the Witten-Reshetikhin-Turaev topological quantum field theory \cite{ReshTur, BHMV, TuraevBook}. An easy variation \cite{BonWon5} of the arguments of \cite{RobertsIrred} shows that this representation $\rho_{\mathrm{WRT}}$ is irreducible. In \cite{BonWon5}, we show that its  classical shadow  in $\RR$ is the character corresponding to the trivial homomorphism $\pi_1(S) \to \{\Id\} \subset \SL(\C)$.

The key to the proof of Theorems~\ref{thm:ShadowIntroA=-1} and \ref{thm:ShadowIntroA=+1} is that, for every knot $K\subset S \times[0,1]$ with no crossing and with vertical framing, the evaluation $T_N\bigl([K]\bigr)$ of the Chebyshev polynomial $T_N$ at the element $[K]\in \SSS$ is central in $\SSS$. 

When the surface $S$ is non-compact, there are central elements that are easier to identify, and provide more invariants. Let $P_k$ be a small simple loop that goes around the $k$--th puncture in $S$. Consider $P_k$ as a knot in $S\times [0,1]$, and endow it with the vertical framing. It is immediate that $[P_k]\in \SSS$ is central in $\SSS$, and the following result easily follows. 

\begin{prop}
\label{prop:PunctInvIntro}
 For every finite-dimensional irreducible representation $\rho \colon \SSS \to \End(V)$, there exists a number $p_k\in \C$ such that
$
\rho\bigl([P_k]\bigr) = p_k\, \Id_V
$.

Suppose in addition that $A^2$ is a primitive $N$--root of unity with $N$ odd, and that  $r_\rho\in \RR$ or $\RS$ is the classical shadow associated to $\rho$ by Theorems~{\upshape\ref{thm:ShadowIntroA=-1}} or {\upshape\ref{thm:ShadowIntroA=+1}}. Then,
$
T_N (p_k) = -\Tr\, r_\rho(P_k)
$.
\end{prop}
The numbers $p_1$, $p_2$, \dots, $p_s$ thus associated to the representation $\rho$ are the \emph{puncture invariants} of $\rho$. The second part of Proposition~\ref{prop:PunctInvIntro} shows that, up to finitely many choices, these puncture invariants are essentially determined by the classical shadow of $\rho$. 

Thus, we have extracted two types of invariants from  a finite-dimensional irreducible representation $\rho\colon \SSS \to \End(V)$  of the skein algebra:  the classical shadow of $\rho$ in $\RR$ or $\RS$, according to whether $A^N=-1$ or $+1$;  the puncture invariants $p_k\in \C$.

The articles  \cite{BonWon3, BonWon4}, which are the natural continuation of this one (see also the expository article \cite{BonWon2}), provide a converse statement.  More precisely, suppose that $A^2$ is a primitive $N$--root of unity with $N$ odd and  that we are given the following data:  a character $r \in \RR$ if $A^N=-1$, or  a twisted character $r\in\RS$ if $A^N = +1$;  a number $p_k\in \C$ for each puncture of $S$ such that  $T_N(p_k) = -\Tr\, r(P_k)$, where $P_k$ is a small loop going around the puncture endowed with the vertical framing. Then \cite{BonWon3, BonWon4} provide a finite-dimensional irreducible representation $\rho\colon \SSS \to \End(V)$ whose classical shadow is equal to $r$ and whose puncture invariants are equal to the $p_k$.

The construction of irreducible representations in \cite{BonWon3, BonWon4} uses the {quantum trace homomorphism} $\Tr_\lambda^\omega$  constructed in \cite{BonWon1}, and defined when the surface $S$ has at least one puncture and negative Euler characteristic. (The article \cite{BonWon4} uses punctured surfaces as a tool to construct representations of  the skein algebras of closed surfaces). For such a punctured surface $S$  and for $\omega=A^{-\frac12}$, the \emph{quantum trace homomorphism} $\Tr_\lambda^\omega\colon\SSS \to \TT$ embeds the skein algebra $\SSS$ in an incarnation  $\TT$  of the quantum Teichm\"uller space of Chekhov, Fock \cite{Foc, CheFoc1, CheFoc2} and Kashaev \cite{Kash}, associated to an ideal triangulation $\lambda$ of $S$. In the special case where $A=1$, the quantum trace $\Tr^{\pm1}_\lambda \bigl ( [K] \bigr) $ is just the Laurent polynomial expressing, for $r\in \RS$, the trace $\Tr\,r(K)$ in terms of suitable square roots $z_i = \sqrt{x_i}$ of the shear-bend coordinates  of $r$ with respect to $\lambda$ (if these exist); a similar property holds when $A=-1$, connecting $\Tr^{\pm\mathrm i}_\lambda \bigl ( [K] \bigr) $ to the classical  trace $\Tr\,r(K)$ for $r\in \RR$.  See \cite[\S1]{BonWon1}. 

The algebraic structure of the quantum Teichm\"uller space $\TT$ is relatively simple, and its finite-dimensional irreducible representations are easily classified \cite{BonLiu}. 
 Composing these representations $\TT \to \End(V)$ with the quantum trace homomorphism $\Tr_\lambda^\omega\colon\SSS \to \TT$, one obtains   many representations $\SSS \to \End(V)$. 
 
 This motivates the second part of the present article (in addition to the fact that this second part is used to prove the results of the first part). When constructing representations of the skein algebra from representations of the quantum Teichm\"uller space, the challenge is to control the classical shadow of the representations so constructed. Recall that this classical shadow is defined using the Chebyshev polynomial $T_N$. 
It turns out that the quantum trace homomorphism is extremely well behaved with respect to $T_N$.  This is expressed in the following result, which  provides  more miraculous cancellations.

\begin{thm}
\label{thm:ChebFrobTraceIntro}
Let the surface $S$ be non-compact and with negative Euler characteristic. 
Suppose that $A^2$ is a primitive $N$--root of unity with $N$ odd, and set $\epsilon = A^N=\pm1$ and $\iota=\sqrt\epsilon\in \{ \pm1, \pm\mathrm i\}$. Then, for every knot $K\subset S \times [0,1]$ whose projection to $S$ has no crossing and with vertical framing, the element
$
  \Tr_\lambda^\omega   \bigl(  T_N ([K]) \bigr)\in \TT
$
is obtained from the classical trace polynomial $\Tr_\lambda^\iota\bigl ([K]\bigr)\in \TTT$ (corresponding to the case where $A$ is replaced by $\epsilon=\pm1$) by replacing each generator $Z_i$ by $Z_i^N$. 
\end{thm}

Since $\Tr_\lambda^\omega$ is an algebra homomorphism, $ \Tr_\lambda^\omega   \bigl(  T_N ([K]) \bigr)$ is also equal to  the evaluation $ T_N \bigl(  \Tr_\lambda^\omega  ([K]) \bigr)$ of the Chebyshev polynomial $T_N$ over the quantum trace $   \Tr_\lambda^\omega  ([K]) $. 

When $A$ is not a root of unity,  $ \Tr_\lambda^\omega   \bigl(  T_N ([K]) \bigr)$ is a Laurent polynomial in non-commuting variables $Z_i$ and, as $N$ tends to $\infty$, the number of its terms grows  as $cN^k$ for appropriate constants $c$, $k>0$. A consequence of Theorem~\ref{thm:ChebFrobTraceIntro} is that, if we specialize $A$ to an appropriate root of unity, most of these terms disappear and we are just left with a constant number of monomials, all of which involve only $N$--th powers of the generators $Z_i$. 

To some extent, the miraculous cancellations of Theorem~\ref{thm:ChebFrobTraceIntro} are even more surprising than those of Theorems~\ref{thm:ShadowIntroA=-1} and \ref{thm:ShadowIntroA=+1}, which at least are conceptually explained by the fact that certain elements of $\SSS$ are central. We use  special cases of Theorem~\ref{thm:ChebFrobTraceIntro} to prove  Theorems~\ref{thm:ShadowIntroA=-1} and \ref{thm:ShadowIntroA=+1}. Our proof of Theorem~\ref{thm:ChebFrobTraceIntro}  essentially relies on brute force, and is quite unsatisfactory in this regard. It seems that this result is quite likely to be grounded in deeper facts about the representation theory of the quantum group $\mathrm U_q(\mathfrak{sl}_2)$. 

While this paper was under review, Thang L\^e provided a simpler proof \cite{Le} of Theorems~\ref{thm:ShadowIntroA=-1} and \ref{thm:ShadowIntroA=+1} (or more precisely of Theorems~\ref{thm:ChebyshevCentral} and \ref{thm:ChebSkeinRelation}, which are the key properties underlying these two statements) that uses only  skein calculus techniques, and in particular does not rely on Theorem~\ref{thm:ChebFrobTraceIntro}. 

\medskip

The authors are grateful to the referees for several useful suggestions. See in particular Remark~\ref{rem:NewSkeinDefinition}. 

\section{Constants}
The article involves various fractional powers of a non-zero complex number $q= \E^{2\pi\I\hbar}\in \C$. Systematically writing these constants as powers of $q$ would lead to somewhat cumbersome notation, such as $\mathcal S^{q^{-\frac12}}\kern -2pt(S)$ or $\mathcal T^{q^{\frac14 N^2}}\kern -4pt(\lambda)$.  It is more convenient (and consistent with the weight of history) to introduce a few more constants related to $q$. While we will regularly remind the reader of the definition of these constants, we collect these definitions here so that this section can be used as an easy reference if needed. 

Again, $q\in \C-\{0\}$ is a non-zero complex number. We choose successive square roots   $A=\sqrt{q^{-1}}$ and $\omega =\sqrt{A^{-1}} = q^{\frac14}$.

For most of the article, $q$ is assumed to be a root of unity, with various restrictions.  The strongest restriction used, under which all properties hold, is that $q$ is a primitive $N$--root of unity with $N$ odd. This is sometimes eased to the weaker assumption that $q^2=A^{-4}$ is a primitive $N$--root of unity, with no parity condition on $N$.

Under any of the above hypotheses, we will set $\epsilon = A^{N^2}$ and $\iota=\omega^{N^2}$. Note  that $\epsilon = \pm1$ if $A^{2N}=1$, and that $\iota^2=\epsilon^{-1}$.

\section{The Kauffman bracket skein algebra}
\label{sect:SkeinAlgebra}

Let $S$ be an oriented surface (without boundary) with finite topological type. Namely, $S$ is obtained by removing finitely many points (possibly none) from a compact oriented surface $\bar S$. We consider \emph{framed links} in the thickened surface $S\times [0,1]$, namely unoriented 1--dimensional submanifolds $K\subset S \times [0,1]$ endowed with a continuous choice of a vector transverse to $K$ at each point of $K$. 
A \emph{framed knot} is a connected framed link. 

The following definition provides a convenient way to describe a framing, in particular when representing a link by a picture. If the projection of $K \subset S \times[0,1]$ to $S$ is an immersion, the \emph{vertical framing} for $K$ is the framing that  everywhere is parallel to the $[0,1]$ factor and points towards $1$. 

The \emph{framed link algebra} $\mathcal K(S)$ is the vector space (over $\C$, say) freely generated by the isotopy classes of  all framed links $K \subset S \times [0,1]$.  
This vector space $\mathcal K(S)$ can be endowed with a multiplication, where the product of $K_1$ and $K_2$ is defined by the framed link $K_1   K_2 \subset S\times[0,1]$ that is the union of $K_1$ rescaled in $S\times [0, \frac12]$ and $K_2$ rescaled in $S\times [\frac12, 1]$. In other words, the product $K_1  K_2$ is defined by superposition of the framed links $K_1$ and $K_2$. 
This \emph{superposition operation} is compatible with isotopies, and therefore provides a well-defined algebra structure on $\mathcal K(S)$. 

Three framed links $K_1$, $K_0$ and $K_\infty$ in $S\times[0,1]$ form a \emph{Kauffman triple} if the only place where they differ is above a small disk in $S$, where they are as represented in Figure~\ref{fig:SkeinRelation} (as seen from above) and where the framing is vertical and pointing upwards (namely the framing is parallel to the $[0,1]$ factor and points towards $1$). 

The \emph{trivial framed knot} is represented by the boundary $O$ of a disk embedded in $S\times\{\frac12\}$, with vertical framing.

For $A\in \C-\{0\}$, the \emph{Kauffman bracket skein algebra} $\SSS$ is the quotient of the framed link algebra $\mathcal K(S)$ by the ideal generated by:
\begin{itemize}
\item all elements $K_1 - A^{-1}K_0 - A K_\infty$ as $(K_1, K_0, K_\infty)$ ranges over all Kauffman triples;
\item the element $O + (A^2 +A^{-2})[\varnothing]$ where $O$ is the trivial framed knot and $\varnothing$ is the empty link. 
\end{itemize} 
The superposition operation descends to a multiplication in $\SSS$, endowing $\SSS$ with the structure of an algebra. The class $[\varnothing]$ of the empty link is an identity element in $\SSS$. 

 An element $[K]\in \SSS$, represented by a framed link $K \subset S \times [0,1]$, is a \emph{skein} in $S$. The construction is defined to ensure that the \emph{skein relation}
 $$
 [K_1] = A^{-1}[K_0] + A [K_\infty]
 $$
 holds in $\SSS$ for every Kauffman triple $(K_1, K_0, K_\infty)$, while
 $$
 [K\cup O]= -(A^2 +A^{-2}) [K]
 $$
 whenever $O$ is the boundary of a disk disjoint from the link $K$ and is endowed with a framing transverse to that disk.

\begin{rem}
\label{rem:NewSkeinDefinition}
The relation $[K\cup O]= -(A^2 +A^{-2}) [K]$  is almost a consequence of the skein relation. Earlier versions of this article (as well as \cite{BonWon1, BonWon2}) omitted it in the definition of the skein algebra, as this convention seemed more natural. The resulting skein algebra $\widehat{\mathcal S}^A(S)$ is only mildly larger, as there is a natural linear isomorphism $\widehat{\mathcal S}^A(S) \cong \SSS \oplus  \R$ where the $\R$ factor is generated by $ [O] + (A^2 +A^{-2})[\varnothing]$ for the trivial framed knot $O$; indeed, a simple string manipulation shows that  $ \bigl( [O]  + A^2 +A^{-2} \bigr) [K]=0$ in $\widehat{\mathcal S}^A(S)$ for every  non-empty link $K$ (see for instance \cite[Lemmas~3.2 and 3.3]{Lick}). While going in this way against the weight of tradition seemed more natural and pedagogic, a particularly insightful referee made us observe that a key ingredient in our arguments, the injectivity of the quantum trace map $\Tr_\lambda^\omega \colon \SSSSS(S) \to \TT$ (and the corresponding statement  \cite[Proposition~29]{BonWon1}) does not  hold if the skein algebra $\SSS$ is replaced by $\widehat{\mathcal S}^A(S)$. We have consequently reverted to the more traditional convention. 
\end{rem}

\section{Central elements of the skein algebra}
\label{sect:CentralElts}

\subsection{Central elements coming from the punctures}
The simplest central elements of the skein algebra $\SSS$ come from the punctures of $S$, if any.

\begin{prop}
\label{prop:PunctCentralElts}
If $P_k$ is a small embedded loop  in $S$ going once around the $k$--th puncture. Considering $P_k$ as a knot in $S\times[0,1]$ endowed with the vertical framing, the skein $[P_k]$ is central in $\SSS$, for every value of $A\in \C-\{0\}$. \qed
\end{prop}

These are somewhat obvious elements of the center of $\SSS$. When $A$ is a root of unity, this center contains other elements which are much less evident, defined using Chebyshev polynomials. 

\subsection{Chebyshev polynomials}
\label{sect:Chebyshev}
Chebyshev polynomials of the first and second kind play a prominent r\^ole in this article.  Both types are polynomials $P_n(x)$ that satisfy the recurrence relation
$$
P_n(x) = xP_{n-1}(x) - P_{n-2}(x).
$$
The (normalized) \emph{Chebyshev polynomials of the first kind} are the polynomials $T_n(x)$ defined by this recurrence relation and by the initial conditions that $T_0(x)=2$ and $T_1(x)=x$. The (normalized)  \emph{Chebyshev polynomials of the second kind} $S_n(x)$ are similarly defined by the above recurrence relation and by the same initial condition $S_1(x)=x$, but differ in the other initial condition $S_0(x)=1$. 

For instance, $T_2(x)=x^2 -2$, $T_3(x) = x^3-3x$ and $T_4(x) = x^4-4x^2+2$,  while $S_2(x)=x^2-1$, $S_3(x)=x^3-2x$ and $S_4(x)=x^4-3x^2+1$. 

One reason why the occurrence of Chebyshev polynomials in this work is not completely unexpected is that both types are closely related to the trace function in $\SL(\C)$, through the following classical properties.

\begin{lem}
\label{lem:Chebyshev}
For any $M\in \SL(\C)$, 
\begin{enumerate}
\item $ \Tr \,M^n = T_n( \Tr\, M) $;
\item  If $\rho_n\colon \SL(\C) \to \mathrm{GL}_{n+1}(\C)$ is the unique $(n+1)$--dimensional irreducible representation of $\SL(\C)$, then $\Tr\, \rho_n(M) = S_n (\Tr\, M)$. \qed
\end{enumerate}
\end{lem}

Applying Lemma~\ref{lem:Chebyshev} to the matrix $M= \begin{pmatrix}
\E^{\I\theta}&0\\0& \E^{-\I\theta}
\end{pmatrix}$ yields the more classical relations that $\cos n\theta = \frac12 T_n(2\cos\theta)$ and $\sin n\theta = \sin \theta\, S_{n-1}(2\cos \theta)$. 

For future reference, we note the following two elementary properties.

\begin{lem} \label{lem:FirstSecondChebyshevs}
For $n\geq 2$, 
 $$T_n(x) = S_n(x) - S_{n-2} (x) .$$
\end{lem}
\begin{proof}
The difference $T_n(x) -S_n(x) + S_{n-2} (x)$  satisfies the linear recurrence relation $P_n(x) = xP_{n-1}(x) - P_{n-2}(x)$, and is equal to $0$ for $n=2$ and $n=3$. 
\end{proof}

\begin{lem}
\label{lem:ChebElementaryProp}
$ $
\begin{enumerate}
\item If $x=a+a^{-1}$, then $T_n(x) = a^n + a^{-n}$;
\item If $y=b+b^{-1}$, the set of solutions to the equation $T_n(x)=y$ consists of the numbers $x=a+a^{-1}$ as $a$ ranges over all $n$--roots of $b$. 
\end{enumerate}
\end{lem}
\begin{proof}
For a matrix $M\in \SL(\C)$, the data of its trace $x$ is equivalent to the data of its spectrum $\{a, a^{-1}\}$. The first property is then an immediate consequence of the fact that $\Tr\, M^n = T_n (\Tr\,M)$. The second property then follows. 
\end{proof}

\subsection{Threading a Chebyshev polynomial along a framed link}
\label{sect:ThreadCheb}

Chebyshev polynomials of the first kind surprisingly provide central elements of the skein algebra $\SSS$.

We first introduce some notation. Consider a polynomial
$$
P(x) = \sum_{i=0}^n a_i x^i. 
$$
We can then associate to each  framed knot $K$ in $S\times [0,1]$ the linear combination
$$
[K^P] = \sum_{i=0}^n a_i\, [K^{(i)}] \in \SSS
$$
where, for each $i$, $K^{(i)}$ is the framed link obtained by taking $i$ parallel copies of $K$ in the direction indicated by the framing.

More generally, if $K$ is a framed link with components $K_1$, $K_2$, \dots, $K_l$, 
$$
[K^P] = \sum_{0\leq i_1, \,i_2,\dots, \,i_l \leq n} a_{i_1} a_{i_2} \dots a_{i_l}\, [K_1^{(i_1)} \cup K_2^{(i_2)} \cup \dots \cup K_l^{(i_l)} ].
$$

We will say that the element $[K^P]\in \SSS$ is obtained by \emph{threading the polynomial $P$ along the framed link $K$}. 

When $K\subset S \times [0,1]$ projects to a simple closed curve in $S$ and is endowed with the vertical framing, $[K^P]$ is equal to the evaluation $P\bigl([K]\bigr)\in \SSS$ of the polynomial $P$ at the skein $[K]$, for the algebra structure of $\SSS$. More generally, if $K$, with components $K_1$, $K_2$, \dots, $K_l$, projects to an embedded submanifold of $S$ and is endowed with the  vertical framing, then $[K^P] = P\bigl([K_1]\bigr) P\bigl([K_2]\bigr) \dots P\bigl([K_l]\bigr)$. Beware that this is in general false for a knot or link whose projection to $S$ admits  crossings, or whose framing is different from the vertical framing.

\begin{thm}
\label{thm:ChebyshevCentral}
If $A^2$ is a primitive $N$--root of unity,  threading the Chebyshev polynomial $T_N$ along a framed link  produces a central element of the skein algebra, namely  $[K^{T_N}]$ is central in $\SSS$ for every framed link $K \subset S \times [0,1]$.  
\end{thm}

Theorem~\ref{thm:ChebyshevCentral} holds with no parity condition on $N$. 

The key to the proof of Theorem~\ref{thm:ChebyshevCentral} is the following  special case in the once-punctured torus $T$, represented in Figure~\ref{fig:PuncturedTorus} as a square with its corners removed and with opposite sides identified. 

\begin{lem}
\label{lem:CentralPuncturedTorus}
In the once-punctured torus $T$, let $L_0$ and $L_\infty$ be the two curves represented in Figure~{\upshape\ref{fig:PuncturedTorus}}, and consider these curves as framed knots with vertical framing in $T\times[0,1]$. If $A^2$ is a primitive $N$--root of unity, then
$$
[L_0^{T_N}] [L_\infty] = [L_\infty][L_0^{T_N}] 
$$
in the skein algebra $\mathcal S^A(T)$. 
\end{lem}

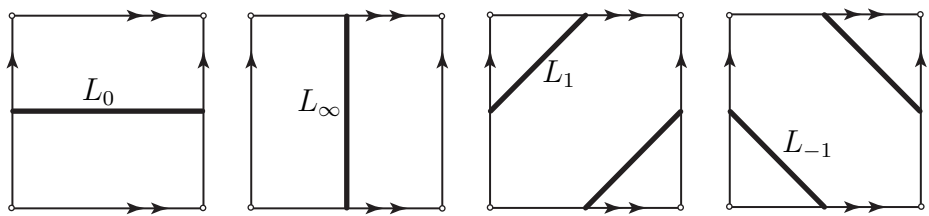
\begin{figure}[htbp]

\SetLabels
(.1 * .55) $L_0 $ \\
( .34*  .5) $L_\infty $ \\
(.6 *  .64) $ L_1$ \\
(.87 *  .3) $L_{-1} $ \\
\endSetLabels
\centerline{\AffixLabels{\includegraphics{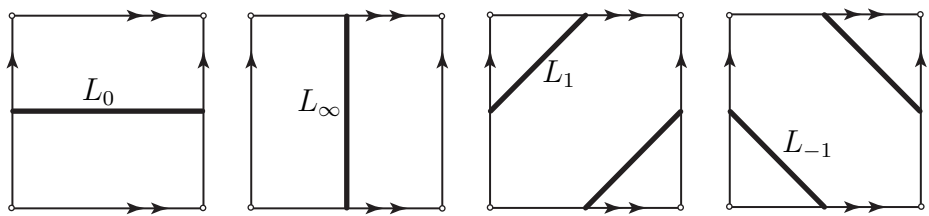}}}

\caption{The curves $L_0$, $L_\infty$, $L_1$ and $L_{-1}$ in the once-punctured torus}
\label{fig:PuncturedTorus}
\end{figure}

There is an elementary proof of Lemma~\ref{lem:CentralPuncturedTorus}, borrowed from the proof of the Product-to-Sum Formula of \cite{FroGel} for the non-punctured torus. See also the (much less elementary) combination of \cite[\S2]{BullPrz} and \cite[Lemma~2]{HavPos}. We provide another proof in \S \ref{sect:ChebPuncTorusSphere}, where it is proved at the same time as Lemma~\ref{lem:ChebSkeinRelationPuncTorus} below.

\begin{proof}[Proof of Theorem~{\upshape\ref{thm:ChebyshevCentral}}, assuming Lemma~{\upshape\ref{lem:CentralPuncturedTorus}}]
We need to show that $[K^{T_N}][L]= [L] [K^{T_N}]$ for any two framed links $K$ and $L$.

 Let $K_1\subset S \times [0,1]$ be isotopic to $K$, and contained in a small neighborhood of $S\times\{0\}$ so that $[K_1^{T_N} \cup L]=[K^{T_N}][L]$ in $\SSS$ (where $[K_1^{T_N} \cup L]$ denotes the element of $\SSS$ obtained by threading the Chebyshev polynomial $T_N$ along the components of $K_1$, and leaving $L$ untouched).  
 
 By progressively changing the undercrossings of $K$ with $L$ to overcrossings, we construct a sequence of framed links $K_1$, $K_2$, \dots, $K_n$  such that each $K_{i+1}$ is obtained from $K_i$ by an isotopy crossing $L$ exactly once, and such that $K_n$ can be isotoped without crossing $L$ into a small neighborhood of $S\times\{1\}$. In particular,  $[K_n^{T_N} \cup L]=[L][K^{T_N}]$ in $\SSS$. It therefore suffices to show that $[K_i^{T_N} \cup L] = [K_{i+1}^{T_N} \cup L]$ for every $i$. 

We now consider an embedded thickened once-punctured torus which ``swallows''  the crossing between $K_i$ and $K_{i+1}$ in question and ``follows'' the two curves otherwise.  We then apply Lemma~\ref{lem:CentralPuncturedTorus} to exchange the swallowed undercrossing to an overcrossing. More precisely, by construction of $K_{i+1}$ from $K_i$ there is an embedding $\phi\colon T\times [0,1] \to S\times[0,1]$ for which, with the notation of Lemma~\ref{lem:CentralPuncturedTorus},  the intersections $K_i\cap\phi \bigl( T \times [0,1] \bigr)$, $L\cap\phi \bigl( T \times [0,1] \bigr)$  and $K_{i+1}\cap\phi \bigl( T \times [0,1] \bigr)$ respectively correspond to the framed knots $L_\infty \times \{ \frac14\}$, $L_0\times \{\frac12\}$ and $L_\infty \times \{\frac34\}$ in $T\times [0,1]$,  after isotopy in $T\times [0,1]$. Applying Lemma~\ref{lem:CentralPuncturedTorus}, and composing with the homomorphism $\mathcal S^A(T) \to \SSS$ induced by $\phi$, shows that $[K_i^{T_N} \cup L] = [K_{i+1}^{T_N} \cup L]$.

By iteration, this concludes the proof. 
\end{proof}

\section{The Chebyshev Homomorphism}
\label{sect:ChebHom}

\begin{thm}
\label{thm:ChebSkeinRelation}
If $A^4$ is a primitive $N$--root of unity and if $\epsilon = A^{N^2} $, there is a unique algebra homomorphism $\mathbf T^A \colon \SSSS \to \SSS$ which, for every framed link $K$ in $S\times [0,1]$,  associates $[K^{T_N}]\in \SSS$ to the skein $[K]\in \SSSS$. 
\end{thm}

A key step in the proof of Theorem~\ref{thm:ChebSkeinRelation} is the following computation. 

\begin{lem}
\label{lem:ChebSkeinRelationPuncTorus}
Suppose that $A^4$ is a primitive $N$--root of unity. 
In the once-punctured torus $T$, let $L_0$, $L_\infty$, $L_1$ and $L_{-1}$ be the curves represented in Figure~{\upshape\ref{fig:PuncturedTorus}}. Considering these curves as knots in $T\times [0,1]$ and endowing them with the vertical framing,
$$
[L_0^{T_N}] [L_\infty^{T_N}] = A^{-N^2} [L_1^{T_N}]+A^{N^2}  [L_{-1}^{T_N}]
$$
in the skein algebra $\mathcal S^A(T)$. 
\end{lem}

We postpone the proof of  Lemma~\ref{lem:ChebSkeinRelationPuncTorus} to \S \ref{sect:ChebPuncTorusSphere}. This result is very reminiscent of the Product-to-Sum Formula of \cite{FroGel}, which holds in the unpunctured torus but for all values of $A$. Unlike that formula, Lemma~\ref{lem:ChebSkeinRelationPuncTorus} only holds when $A^4$ is an $N$--root of unity. See the recent preprint \cite{Le} for a simpler proof of  Lemma~\ref{lem:ChebSkeinRelationPuncTorus}. 

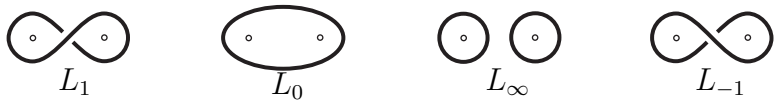
\begin{figure}[htbp]
\SetLabels
( .1*-.3 ) $L_1 $ \\
(.37 * -.35) $L_0 $ \\
(.65 *-.3 ) $ L_\infty$ \\
( .92*-.3 ) $L_{-1} $ \\
\endSetLabels
\centerline{\AffixLabels{ \includegraphics{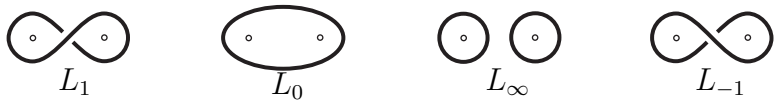} }}
\vskip 10pt
\caption{The $1$--submanifolds $L_1$, $L_0$, $L_\infty$ and $L_{-1}$  in the twice-punctured plane}
\label{fig:PuncturedSphere}
\end{figure}

The proof of Theorem~\ref{thm:ChebSkeinRelation} also uses the following similar computation. 

\begin{lem}
\label{lem:ChebSkeinRelationPuncSphere}
Suppose that $A^4$ is a primitive $N$--root of unity. 
For the twice-punctured plane $U$, let $L_1$, $L_0$, $L_\infty$  and  $L_{-1}$ be the links in $U\times[0,1]$ represented in Figure~{\upshape\ref{fig:PuncturedSphere}}, endowed  with the vertical framing. Then,
\begin{align*}
[L_1^{T_N}]&= A^{-N^2} [L_0^{T_N}]+A^{N^2}  [L_\infty^{T_N}]\\
\text{and } [L_{-1}^{T_N}]&= A^{N^2} [L_0^{T_N}]+A^{-N^2}  [L_\infty^{T_N}]
\end{align*}
in the skein algebra $\mathcal S^A(U)$. 
\end{lem}

Again,  we  postpone the proof of  Lemma~\ref{lem:ChebSkeinRelationPuncSphere} to \S \ref{sect:ChebPuncTorusSphere}.

\begin{proof}
[Proof of Theorem~{\upshape\ref{thm:ChebSkeinRelation}} (assuming Lemmas~{\upshape\ref{lem:ChebSkeinRelationPuncTorus}} and {\upshape\ref{lem:ChebSkeinRelationPuncSphere}})]

We have to check that the Chebyshev threads $[K^{T_N}]\in \SSS$ satisfy the skein relation with $A$ replaced by $\epsilon=A^{N^2}$, and that $[O^{T_N}]=-\epsilon^2- \epsilon^{-2}$ for the trivial framed knot $O$. 

We begin with the easier of these statements, namely the second one. First of all, note that $\epsilon^2 = A^{2N^2}=1$, so we really have to prove that $[O^{T_N}]= -2$. Because $O$ has no crossing and is endowed with the vertical framing,
$$
[O^{T_N}] = T_N\bigl( [O] \bigr) = T_N( -A^2 -A^{-2}) =  -A^{2N}-A^{-2N} = -2
$$
as required, using Lemma~\ref{lem:ChebElementaryProp} and the fact that $N$ is odd for the third equality.

We now turn to the skein relation. Let the framed links $K_1$, $K_0$, $K_\infty \subset S \times[0,1]$ form a Kauffman triple, as in Figure~\ref{fig:SkeinRelation}. Let $\widehat K_1$, $\widehat K_0$, $\widehat K_\infty $ denote the union of the (one or two) components of $K_1$, $K_0$, $K_\infty$, respectively,  that appear in Figure~\ref{fig:SkeinRelation}.  (There may be additional components that do not appear on the picture.) Considering the way the strands represented connect outside of the picture, we distinguish three cases, according to the number of components of $\widehat K_1$, $\widehat K_0$ and $\widehat K_\infty $. One easily sees that exactly two of these three links are connected, while the remaining one has two components. 

If $\widehat K_1$ is disconnected, there exists an embedding $\phi \colon T\times [0,1] \to S \times [0,1]$ of the thickened punctured torus $T\times [0,1]$ such that, using the notation of Lemma~\ref{lem:ChebSkeinRelationPuncTorus},  $\phi^{-1}(K_1) $ is isotopic to the union of $L_0 \times \{ \frac 14\}$ and of $L_\infty \times \{ \frac 34\}$ in $T\times [0,1]$,  while $\phi^{-1}(K_0)$ and $\phi^{-1}(K_\infty) $ are respectively isotopic to $L_1 \times \{\frac12\}$ and $L_{-1} \times \{\frac12\}$.   Lemma~\ref{lem:ChebSkeinRelationPuncTorus} shows that 
$$
[\widehat K_1^{T_N}] = A^{-N^2}[\widehat K_0^{T_N}] + A^{N^2}[\widehat K_\infty^{T_N}] = \epsilon^{-1} [\widehat K_0^{T_N}] + \epsilon [\widehat K_\infty^{T_N}]
$$
in $\mathcal S^A \bigl( \phi \bigl(T\times[0,1]\bigr)\bigr)$. 
Since $K_1$, $K_0$, $K_\infty $ coincide outside of $\phi\bigl(T\times[0,1]\bigr)$, it follows that 
$$
[K_1^{T_N}] = \epsilon^{-1}  [K_0^{T_N}] + \epsilon[K_\infty^{T_N}] .
$$
in $\SSS$. 

The other two cases are similar, but using Lemma~\ref{lem:ChebSkeinRelationPuncSphere} this time. If $\widehat K_\infty$ is disconnected, then there exists an embedding $\phi \colon U\times [0,1] \to S \times [0,1]$ of the thickened twice-punctured plane $U\times [0,1]$ such that, using the notation of Lemma~\ref{lem:ChebSkeinRelationPuncSphere},  $\phi^{-1}(K_1)= L_1$,  $\phi^{-1}(K_0) = L_0$ and $\phi^{-1}(K_\infty) = L_\infty$. Applying the first identity of Lemma~\ref{lem:ChebSkeinRelationPuncSphere} shows that 
$$
[\widehat K_1^{T_N}] =  \epsilon^{-1} [\widehat K_0^{T_N}] + \epsilon [\widehat K_\infty^{T_N}]
$$
in this case as well, so that $[K_1^{T_N}] =  \epsilon^{-1} [K_0^{T_N}] + \epsilon [K_\infty^{T_N}]
$ in $\SSS$. 

Finally, if $\widehat K_0$ is disconnected, we use an embedding $\phi \colon U\times [0,1] \to S \times [0,1]$ of the thickened twice-punctured plane $U\times [0,1]$ such that  $\phi^{-1}(K_1)= L_{-1}$,  $\phi^{-1}(K_0) = L_\infty$ and $\phi^{-1}(K_\infty) = L_0$. The second identity of Lemma~\ref{lem:ChebSkeinRelationPuncSphere} then provides the relation sought.

Therefore, 
$
[K_1^{T_N}] =  \epsilon^{-1} [K_0^{T_N}] + \epsilon [K_\infty^{T_N}]
$ in $\SSS$ 
for every Kauffman triple $K_1$, $K_0$, $K_\infty$. This proves that the threading map $K \mapsto [K^{T_N}]\in \SSS$ induces a linear map $\mathbf T^A \colon \SSSS \to \SSS$. It is immediate that $\mathbf T^A$ is an algebra homomorphism. 
\end{proof}

\section{Invariants of finite-dimensional irreducible representations}
\label{sect:RepInvariants}

Let $\rho\colon \SSS \to \End(V)$ be a finite-dimensional irreducible  representation of the skein algebra $\SSS$. 
 We want to construct invariants for $\rho$. 

To take advantage of Theorems~\ref{thm:ChebyshevCentral} and \ref{thm:ChebSkeinRelation}, we are  going to need   that both $A^4$ is a primitive $N$--root of unity and $A^{2N}=1$. This is equivalent to the property that $A^2$ is a primitive $N$--root of unity with $N$ odd. In particular, $\epsilon=A^{N^2}=A^N=\pm1$.

\subsection{The classical shadow}
\label{sect:Shadow}

Composing the representation  $\rho\colon \SSS \to \End(V)$  with the Chebyshev homomorphism $\mathbf T^A \colon \SSSS \to \SSS$ of Theorem~\ref{thm:ChebSkeinRelation}, we obtain a homomorphism $\SSSS \to \End(V)$. By Theorem~\ref{thm:ChebyshevCentral}, this homomorphism factors through the center of $\SSS$ and therefore, by irreducibility of $\rho$ and by Schur's lemma, there exists an algebra homomorphism
$$
\kappa_\rho \colon \SSSS \to \C 
$$
such that 
$$
\rho \circ \mathbf T^A\bigl ([K]\bigr) = \rho\bigl([K^{T_{N} }] \bigr) = \kappa_\rho \bigl([K] \bigr) \Id_V
$$
for every skein $[K]\in\SSS$.

Because  $\epsilon = \pm1$,  such a homomorphism $\kappa_\rho \colon \SSSS \to \C $ has a geometric interpretation. 

More precisely, when $\epsilon = -1$, consider the character variety
$$
\RR = \{ \text{group homomorphisms } r\colon \pi_1(S) \to \SL(\C) \} \db \SL(\C)
$$
where $\SL(\C)$ acts  on homomorphisms by conjugation, and where the double bar $\db$ indicates that the quotient is taken in the sense of geometric invariant theory \cite{Mum} in algebraic geometry. For a group homomorphism $r\colon \pi_1(S) \to \SL(\C)$ and a closed curve $K\subset S \times [0,1]$, the trace $\Tr\,r(K)\in \C$ depends only on the class of $r$ in the character variety $\RR$. An observation of Doug Bullock, Charlie Frohman, Jozef Przytycki and Adam Sikora \cite{Bull1, Bull2, BFK1, BFK2, PrzS} then shows that this defines an algebra homomorphism
$$
\Tr_r \colon \mathcal S^{-1}(S) \to \C
$$
by the property that
$$
\Tr_r \bigl( [K] \bigr) = - \Tr\, r(K)
$$
for every (connected) framed knot $K\subset S\times [0,1]$. Note that $\Tr_r \bigl([K] \bigr)$ is independent of the framing of $K$. 

There exists a twisted version of this when $\epsilon = +1$, namely for the skein algebra $\mathcal S^{+1}(S)$. Let $\Spin$ denote the space of isotopy classes of spin structures on $S$. Given a spin structure $\sigma\in \Spin$,  John Barrett \cite{Barr} constructs  an algebra homomorphism $B_\sigma\colon \mathcal S^{+1}(S) \to \mathcal S^{-1}(S)$, which to a skein $[K]\in \mathcal S^{+1}(S)$ associates $(-1)^{k+\sigma(k)} [K]\in \mathcal S^{-1}(S)$, where $k$ is the number of components of  the link $K$ and where $\sigma(K)\in \Z_2$ is the monodromy of the framing of $K$ with respect to the spin structure $\sigma$. A pair $r=(\widehat r,\sigma)$, consisting of  a group homomorphism $\widehat r\colon \pi_1(S) \to \SL(\C)$ and a spin structure $\sigma\in \Spin$, then defines a trace map
$$
\Tr_{r} = \Tr_{\widehat r} \circ B_\sigma \colon  \mathcal S^{+1}(S) \to \C,
$$
which to a connected skein $[K]\in  \mathcal S^{+1}(S)$ associates $$\Tr_r \bigl( [K] \bigr) = (-1)^{\sigma(K)} \Tr\, \widehat r(K). $$ 

This trace map $\Tr_r$ is unchanged under certain modifications of the pair $(\widehat r,\sigma)$. Indeed, two spin structures in $\Spin$  differ by an obstruction in $H^1(S;\Z_2)$; this  defines an action of $H^1(S;\Z_2)$ on $\Spin$. The cohomology group $H^1(S;\Z_2)$ also acts on the character variety $\RR$ by the property that, for a homomorphism $\widehat r \colon \pi_1(S) \to \SL(\C)$  and $\alpha \in H^1(S;\Z_2)$, the character $\alpha \widehat r\in \RR$ is represented by
\begin{align*}
\alpha \widehat r(\gamma)\colon \pi_1(S)  &\longrightarrow\quad  \SL(\C)\\
\gamma \quad&\longmapsto   (-1)^{\alpha(\gamma)} \widehat r(\gamma) .
\end{align*}
Hence the trace map $\Tr_{r}  \colon  \mathcal S^{+1}(S) \to \C$ depends only on the image of $r=(\widehat r,\sigma)$ in  the quotient
$$
\RS = \left( \RR \times \Spin\right) / H^1(S;\Z_2).
$$

To explain the appearance of the group $\PSL(\C)$ in the notation, consider the character variety
$$
\mathcal R_{\PSL(\C)}(S)=
 \{ \text{group homomorphisms } r\colon \pi_1(S) \to \PSL(\C) \} \db \PSL(\C)
$$
and the subset $\mathcal R_{\PSL(\C)}^0(S)= \RR/H^1(S;\Z_2)$ of $\mathcal R_{\PSL(\C)}(S)$ consisting of those homomorphisms $\pi_1(S) \to \PSL(\C)$ that lift to $\SL(\C)$. Bill Goldman \cite{Gold} showed that $\mathcal R_{\PSL(\C)}^0(S)$ is equal to the whole character variety $\mathcal R_{\PSL(\C)}(S)$ when $S$ is non-compact, and is equal to one of the two components of $\mathcal R_{\PSL(\C)}(S)$  when $S$ is compact. Then, $\RS$ can be seen as a twisted product of $\mathcal R_{\PSL(\C)}^0(S)$
with $\Spin$. In particular, both $\RR$ and $\RS$ are coverings of $\mathcal R_{\PSL(\C)}^0(S)$ with fiber $\Spin \cong H^1(S;\Z_2)$. Choosing a spin structure $\sigma \in \Spin$ provides an identification $\RS \cong \RR$, but there is no natural such identification.

Conversely, the following statement shows that all algebra homomorphisms $\kappa \colon \SSSS \to \C$ are obtained in this way, via a trace map. 

\begin{prop}
\label{prop:Helling}
 Every 
 homomorphism $\kappa: \mathcal S^\epsilon (S) \to \C$ is equal to the trace homomorphism $\Tr_r$ associated to a  unique character $r \in \RR$ if $\epsilon =-1$, or a twisted character $r\in \RS$ if $\epsilon=+1$. 

\end{prop}

\begin{proof}
The case $\epsilon =-1$ is a result of Heinz Helling \cite[Prop.~1]{Hell}. 
The case $\epsilon =+1$ easily follows by using the Barrett isomorphism $B_\sigma\colon \mathcal S^{+1}(S) \to \mathcal S^{-1}(S)$ provided by a spin structure $\sigma \in \Spin$.  
\end{proof}

Let us now return to our analysis of a finite-dimensional irreducible representation $\rho \colon \SSS \to \End(V)$, and to its associated homomorphism $\kappa_\rho \colon \SSSS\to \C$ defined by the property that
$ \rho \bigl( [K^{T_{N} }] \bigr) = \kappa_\rho\bigl([K]\bigr) \Id_V
$
for every skein $[K]\in\SSS$.

Combining the definition of the algebra homomorphism  $\kappa_\rho \colon \SSSS \to \C $  with Proposition~\ref{prop:Helling} 
provides the following statement. 

\begin{thm}
\label{thm:Shadow}
Suppose that $A^2$ is a primitive $N$--root of unity with $N$ odd, so that $\epsilon =A^{N^2}$ is equal to $\pm1$.  Then, for every finite-dimensional irreducible representation $\rho \colon \SSS \to \End(V)$ of the skein algebra $\SSS$, there exists a unique  character $r_\rho\in \RR$ if $A^N=-1$, or  a unique twisted character $r_\rho\in \RS$ if $A^N=+1$, such that
$$
\rho\bigl ([K^{T_N}] \bigr) = \bigl( \Tr_{r_\rho}([K]) \bigr) \Id_V
$$ 
for every skein $[K] \in \SSS$.   \qed
\end{thm}

In particular, Theorem~\ref{thm:Shadow} implies Theorems~\ref{thm:ShadowIntroA=-1} and \ref{thm:ShadowIntroA=+1} stated in the introduction since,   for a knot $K\subset S\times [0,1]$ with no crossing and with vertical framing, $T_N\bigl(\rho([K])\bigr) = \rho\bigl([K^{T_N}]\bigr)$ and $ \Tr_{r_\rho}([K]) =-\Tr\, r_\rho(K)$. 

We call the character  $r_\rho\in \RR$ or $ \RS$, associated to $\kappa_\rho$ by Theorem~\ref{thm:Shadow},  the \emph{classical shadow} of the irreducible representation $\rho \colon \SSS \to \End(V)$.

\subsection{Puncture invariants}

The central elements associated to the punctures of $S$ provide similar invariants for the finite-dimensional irreducible representation $\rho\colon \pi_1(S) \to \End(V)$. More precisely, let $P_k\subset S\times \{\frac12\}\subset S\times[0,1]$ be a small simple loop going around the $k$--th puncture, endowed with the vertical framing. Proposition~\ref{prop:PunctCentralElts} shows that the corresponding skein $[P_k]\in \SSS$ is central. Therefore, by irreducibility of $\rho$, there exists a number $p_k \in \C$ such that 
$$
\rho\bigl([P_k]\bigr) = p_k\, \Id_V. 
$$
This number $p_k\in \C$ is the \emph{$k$--th puncture invariant} of the finite-dimensional irreducible representation~$\rho$.

These  numbers $p_k$ are clearly constrained by the classical shadow $r_\rho \in \RR$ or $\RS$ of  $\rho$, in terms of the Chebyshev polynomial $T_N$. Indeed:

\begin{lem}
\label{lem:PunctInvShadow}
$$T_N(p_k)  = -\Tr \, r_\rho(P_k).$$
\end{lem}

\begin{proof}
By  definition of the homomorphism $\Tr_{r_\rho} \colon \SSSS \to \C$ and by Theorem~\ref{thm:Shadow},
\begin{align*}
\bigl(\Tr \, r_\rho(P_k)\bigr)\Id_V &=- \Tr_{r_\rho}\bigl( [P_k] \bigr) \Id_V =- \rho\bigl(\bigl[P_k^{T_{N} }\bigr] \bigr) = -\rho\bigl(T_N([P_k])\bigr)\\
 & =- T_N\bigl(\rho([P_k])\bigr) =- T_N( p_k\Id_V) =- T_N(p_k) \Id_V
\end{align*}
where the third equality uses the fact that $P_k$ has no crossings and has vertical framing, and where the fourth equality holds because $\rho$ is an algebra homomorphism. 
\end{proof}

In particular, Lemma~\ref{lem:ChebElementaryProp} shows that, once the classical shadow $r_\rho$  is given, there are at most $N$ possibilities for each puncture invariant $p_k$.

\section{The quantum Teichm\"uller space and the quantum trace homomorphism}

In \cite{BonWon3, BonWon4}, we prove a converse to Theorem~\ref{thm:Shadow} and Lemma~\ref{lem:PunctInvShadow}, by showing that every character $r\in \RR$ or $\RS$ and any set of puncture weights $p_k$ with  $T_N(p_k) = \Tr_{r_\rho}\bigl( [P_k] \bigr)$ can be realized as the classical shadow and the puncture invariants of a finite-dimensional irreducible representation $\rho\colon \SSS \to \End(V)$. 

The proof of this result relies on the quantum Teichm\"uller space of \cite{Foc, CheFoc1, CheFoc2} and on the quantum trace homomorphism constructed in \cite{BonWon1}. A crucial step in the argument is a compatibility property between the quantum trace homomorphism and the Chebyshev homomorphism, appearing here as Theorem~\ref{thm:ChebyQTracesFrob}. This section is devoted to the precise statement of this result. This property is of interest by itself, because of the surprising ``miraculous cancellations'' that it exhibits. A special case of  Theorem~\ref{thm:ChebyQTracesFrob} also provides a key step in the proof of  Lemmas~\ref{lem:CentralPuncturedTorus}, \ref{lem:ChebSkeinRelationPuncTorus} and \ref{lem:ChebSkeinRelationPuncSphere}, which we had temporarily postponed for this very reason

\subsection{The Chekhov-Fock algebra}
\label{sect:CheFock}

The Chekhov-Fock algebra is the avatar of the quantum Teichm\"uller space associated to an ideal triangulation of the surface $S$. If $S$ is obtained from a compact surface $\bar S$ by removing finitely many points $v_1$, $v_2$, \dots, $v_s$, an
 \emph{ideal triangulation} of $S$ is a triangulation $\lambda$ of $\bar S$ whose vertex set is exactly
$\{ v_1,v_2, 
\dots, v_s \}$. 

Throughout the rest of the article, we will assume that $S$ admits such an ideal triangulation $\lambda$, which is equivalent to the property that $S$ has at least one puncture and that its Euler characteristic is negative. 

Let $e_1$, $e_2$, \dots, $e_n$ denote the edges of $\lambda$. Let $a_{ij} \in \{0,1, 2\}$ be the number of times an end of the edge $e_j$ immediately succeeds an end of $e_i$ when going counterclockwise around a puncture of $S$, and set $\sigma_{ij}=a_{ij}-a_{ji}\in \{-2, -1, 0, 1, 2\}$.  The \emph{Chekhov-Fock algebra} $\TT$ of $\lambda$ is the algebra defined by generators $Z_1^{\pm1}$, $Z_2^{\pm1}$, \dots, $Z_n^{\pm1}$ associated to the edges  $e_1$, $e_2$, \dots, $e_n$ of $\lambda$, and by the relations
$$
Z_iZ_j = \omega^{2\sigma_{ij}} Z_jZ_i.
$$ 

(The actual Chekhov-Fock algebra $\mathcal T^q(\lambda)$ that is at the basis of the quantum Teichm\"uller space uses the constant $q=\omega^4$ instead of $\omega$. The generators $Z_i$ of $\TT$ appearing here are designed to model square roots of the original generators of $\mathcal T^q(\lambda)$.)

An element of the Chekhov-Fock algebra $\TT$ is a linear combination of monomials $Z_{i_1}^{n_1}Z_{i_2}^{n_2} \dots Z_{i_l}^{n_l}$ in the generators $Z_i$, with $n_1$, $n_2$, \dots, $n_l\in \Z$. Because of the skew-commutativity relation $Z_iZ_j = \omega^{2\sigma_{ij}} Z_jZ_i$, the order of the variables in such a  monomial does matter. It is convenient to use the following symmetrization trick. 

The \emph{Weyl quantum ordering} for  $Z_{i_1}^{n_1}Z_{i_2}^{n_2} \dots Z_{i_l}^{n_l}$ is the monomial
$$
[Z_{i_1}^{n_1}Z_{i_2}^{n_2} \dots Z_{i_l}^{n_l}] = \omega^{-\sum_{u<v} n_un_v\sigma_{i_ui_v}} Z_{i_1}^{n_1}Z_{i_2}^{n_2} \dots Z_{i_l}^{n_l}. 
$$
The formula is specially designed so that $[Z_{i_1}^{n_1}Z_{i_2}^{n_2} \dots Z_{i_l}^{n_l}] $ is invariant under any permutation of the $Z_{i_u}^{n_u}$. 

Note that the Chekhov-Fock algebra $\TT$ depends only on $\omega^2$, and that there is a canonical isomorphism $\TT\cong \mathcal T^{-\omega}(\lambda)$. However, the Weyl quantum ordering does depend  on $\omega$, and justifies the emphasis on $\omega$ (as opposed to $\omega^2$) in the notation for $\TT$.

\subsection{The Frobenius homomorphism} 
\label{subsect:Frobenius}

The following homomorphism plays a fundamental role in the classification of finite-dimensional irreducible representations of the quantum Teichm\"uller space \cite{BonLiu}.

\begin{prop}
\label{prop:Frobenius}
If  $\iota = \omega^{N^2}$, there  is a unique algebra homomorphism
$$
\mathbf F^\omega \colon \TTT   \to \TT$$ 
which maps each generator $Z_i\in \TTT$ to $Z_i^N\in \TT$, where in the first instance $Z_i\in \TTT$ denotes the generator associated to the $i$--th edge $e_i$ of $\lambda$, whereas the second time $Z_i\in \TT$ denotes the generator of $\TT$ associated to the same edge~$e_i$. 
\end{prop}

\begin{proof}
If $Z_iZ_j = \omega^{2\sigma_{ij}} Z_jZ_i$ for generators $Z_i$, $Z_j$ of $\TT$, then $Z_i^N Z_j^N \linebreak = \omega^{2\sigma_{ij}N^2} Z_j^N Z_i^N   = \iota^{2\sigma_{ij}} Z_j^N Z_i^N$.
\end{proof}

Borrowing the terminology from \cite{FockG}, the algebra homomorphism $\mathbf F^\omega \colon \TTT   \to \TT$ is the \emph{Frobenius homomorphism} of $\TT$. When $\omega^N = (-1)^{N+1}$, this Frobenius is remarkably well-behaved with respect to coordinate changes in the quantum Teichm\"uller space; see \cite{CheFoc1, BonLiu, FockG}. We will not need this property here, but it somewhat foreshadows the compatibility property of Theorem~\ref{thm:ChebyQTracesFrob}.

\subsection{The quantum trace homomorphism}
\label{sect:QTrace}

Another key ingredient is the quantum trace homomorphism  $\Tr_\lambda^\omega \colon \SSS \to \TT$ constructed in \cite{BonWon1}. 
The computations of \S \ref{sect:ChebyQTraces} will (unfortunately) require familiarity with the details of this embedding of the skein algebra  $\SSS$ into the Chekhov-Fock algebra $\TT$. We therefore need to briefly summarize the main technical features of its construction.

In particular, we need to consider situations where $S$ is a  punctured surface with
boundary,  obtained by removing finitely many points $v_1$,
$v_2$,
\dots, $v_s$ from a compact connected oriented surface
$\bar S$  with  boundary $\partial \bar S$. We require that each
component of $\partial \bar S$ contains at least one puncture $v_k$, that
there is at least one puncture, and that  the Euler characteristic of $S$ is less that half the number of components of $\partial S$. In particular, the boundary $\partial S$ consists of disjoint open intervals. These topological restrictions are
equivalent   to the existence of an ideal triangulation $\lambda$ for $S$. The Chekhov-Fock algebra $\TT$ is then defined as before, with skew-commuting  generators corresponding to the edges of $\lambda$ (including the components of $\partial S$). 

To define the skein algebra for such a punctured surface $S$ with boundary, we take a \emph{framed link} $K$ in $S\times[0,1]$ to be  a 1--dimensional framed submanifold $K \subset S \times [0,1]$ such that:
\begin{enumerate}
\item $\partial K =  K \cap \partial(S \times [0,1])$ consists of finitely many points in $(\partial S) \times [0,1]$;
\item at every point of $\partial K$, the framing is vertical, namely parallel to the $[0,1]$ factor  and  pointing in the direction of $1$;
\item for every component $e$ of $\partial S$, the points of $\partial K$ that are in $e\times [0,1]$ sit at different elevations, namely have different $[0,1]$--coordinates. 
\end{enumerate}

Isotopies of framed links are required to respect all three conditions above. The third condition is particular important for the gluing construction described below. 

The \emph{skein algebra} $\SSS$ is then defined as before, first considering the vector space $\mathcal K(S)$ freely generated by all isotopy classes of framed links, and then taking its quotient by the skein relation and by the relation that $[O]=-A^2 - A^{-2}$ for the trivial framed knot $O$. The multiplication of the  algebra structure is provided by the usual superposition operation. 

The extension to surfaces with boundary enables us to glue surfaces and skeins. 
Given two surfaces $S_1$ and $S_2$ and two boundary components $e_1 \subset \partial S_1$ and $e_2 \subset \partial S_2$, we can glue $S_1$ and $S_2$ by identifying $e_1$ and $e_2$ to obtain a new oriented  surface $S$. There is a unique way to perform this gluing so that the orientations of $S_1$ and $S_2$ match to give an orientation of $S$. We allow the ``self-gluing'' case, where the surfaces $S_1$ and $S_2$ are equal as long as the boundary components $e_1$ and $e_2$ are distinct. If we are given an ideal triangulation $\lambda_1$ of $S_1$ and an ideal triangulation $\lambda_2$ of $S_2$, these two triangulations fit together to give an ideal triangulation $\lambda$ of the glued surface $S$.

Now, suppose in addition that we are given skeins $[K_1] \in \mathcal S^A (S_1)$ and $[K_2]\in\mathcal S^A (S_2)$, represented by framed links $K_1$ and $K_2$  such that $K_1 \cap( e_1 \times[0,1])$ and $K_2 \cap( e_2 \times[0,1])$ have the same number of points. We can then arrange by an isotopy of framed links that $K_1$ and $K_2$ fit together to give a framed link $K\subset  S \times [0,1]$; note that it is here important that the framings be vertical pointing upwards on the boundary, so that they fit together to give a framing of $K$. By our hypothesis that the points of $K_1 \cap( e_1 \times[0,1])$ (and of $K_2 \cap( e_2 \times[0,1])$) sit at different elevations, the framed link $K$ is now uniquely determined up to isotopy. Also, this operation is well behaved with respect to the skein relations, so that $K$ represents a well-defined element $[K] \in \SSS$. 
We say that $[K] \in \SSS$ is \emph{obtained by gluing} the two skeins $[K_1] \in \mathcal S^A (S_1)$ and $[K_2]\in\mathcal S^A (S_2)$. 

The gluing operation is also well-behaved with respect to Chekhov-Fock algebras. Indeed, if $S$ is obtained by gluing two surfaces $S_1$ and $S_2$ along boundary components $e_1 \subset \partial S_1$ and $e_2 \subset \partial S_2$, and if the ideal triangulations $\lambda_1$ of $S_1$ and  $\lambda_2$ of $S_2$ are glued to provide an ideal triangulation $\lambda$ of $S$, there is a canonical embedding $\Phi\colon \TT \to \mathcal T^\omega(\lambda_1) \otimes  \mathcal T^\omega(\lambda_2)$ defined as follows. If $Z_e$ is a generator of $\TT$ associated to an edge $e$ of $\lambda$:
\begin{enumerate}
\item $\Phi(Z_e) = Z_{e_1}\otimes Z_{e_2}$ if $e$ is the edge coming from the boundary components $e_1 \subset \partial S_1$ and $e_2 \subset \partial S_2$ that are glued together, and if $Z_{e_1}\in \mathcal T^\omega(\lambda_1)$ and $Z_{e_2}\in \mathcal T^\omega(\lambda_2)$ are the generators associated to these edges $e_1$, $e_2$ of $\lambda_1$, $\lambda_2$;

\item $\Phi(Z_e) = Z_{e_1'}\otimes 1$ if $e$ corresponds to an edge $e_1'$ of $\lambda_1$ that is not glued to an edge of $\lambda_2$;

\item $\Phi(Z_e) = 1 \otimes Z_{e_2'}$ if $e$ corresponds to an edge $e_2'$ of $\lambda_2$ that is not glued to an edge of $\lambda_1$. 
\end{enumerate}
There is a similar embedding $\TT \to \mathcal T^\omega(\lambda_1)$ in the case of self-gluing, when the surface $S$ and the ideal triangulation $\lambda$ are obtained by gluing two distinct boundary components of a surface $S_1$ and an ideal triangulation $\lambda_1$ of $S_1$. 

Finally,  a \emph{state} for a skein $[K]\in \SSS$ is the assignment $s\colon \partial K \to \{ +, -\}$ of a sign $\pm$ to each point of $\partial K$. Let $\SSSSS(S)$ be the algebra consisting of linear combinations of \emph{stated skeins}, namely of  skeins endowed with states. In particular, $\SSSSS(S) = \SSS$ when $S$ has empty boundary. 

In the case where $K \in \SSS$ is obtained by gluing the two skeins $K_1 \in \mathcal S^A (S_1)$ and $K_2\in\mathcal S^A (S_2)$, the states $s\colon \partial K \to \{ +,-\}$, $s_1\colon \partial K_1 \to \{ +,-\}$, $s_2\colon \partial K_2 \to \{ +,-\}$ are \emph{compatible} if $s_1$ and $s_2$ coincide on $\partial K_1 \cap( e_1 \times[0,1])=\partial K_2 \cap( e_2 \times[0,1])$ for the identification given by the gluing, and if $s$ coincides with the restrictions of $s_1$ and $s_2$ on $\partial K \subset \partial K_1 \cup \partial K_2$.

\begin{thm}[\cite{BonWon1}]
\label{thm:QTrace}
  For $A=\omega^{-2}$, there  is a unique family of  algebra homomorphisms 
$$\Tr_\lambda^\omega \colon \SSSSS(S) \to \TT,$$
 defined for each punctured surface $S$ with boundary  and for each ideal triangulation $\lambda$ of $S$, such that:
\begin{enumerate}
\item {\upshape (State Sum Property)} If the surface $S$ is obtained by gluing $S_1$ to $S_2$, if  the ideal triangulation $\lambda$ of $S$ is obtained by combining the ideal triangulations $\lambda_1$ of $S_1$ and $\lambda_2$ of $S_2$, and if the skeins $[K_1] \in \mathcal S^A (S_1)$ and $[K_2]\in\mathcal S^A(S_2)$ are glued together to give $[K] \in \SSS$, then
$$
\quad \quad \quad
 \Tr_\lambda^\omega \bigl([K,s]\bigr) = \sum_{\text{compatible }s_1, s_2}\Tr_{\lambda_1}^\omega \bigl([K_1, s_1]\bigr) \otimes \Tr_{\lambda_2}^\omega \bigl([K_2, s_2]\bigr) $$
where the sum is over all states $s_1\colon \partial K_1 \to \{+,-\}$ and $s_2\colon \partial K_2 \to \{+,-\}$ that are compatible with $s\colon \partial K \to \{+,-\}$ and with each other. Similarly if the surface $S$, the ideal triangulation $\lambda$ of $S$, and the skein $[K] \in \SSS$ are obtained by gluing  the surface $S_1$, the ideal triangulation $\lambda_1$ of $S_1$, and the skein $[K_1] \in \mathcal S^A(S_1)$, respectively,  to themselves, then
$$
\quad \quad
 \Tr_ \lambda^\omega \bigl([K,s]\bigr) = \sum_{\text{compatible }s_1}\Tr_{\lambda_1}^\omega\bigl([K_1, s_1]\bigr)  .$$

\item {\upshape (Elementary Cases)} When $S$ is a triangle and $K$ projects to a single arc embedded in $S$, with vertical framing, then 
\begin{enumerate}
\item  in the case of Figure~{\upshape\ref{fig:TriangleOneArc}(a)}, where $\epsilon_1$, $\epsilon_2=\pm$ are the signs associated by the state $s$ to the end points of $K$, 
$$
\quad\quad \quad \quad \quad
\Tr_\lambda^\omega \bigl([K, s]\bigr) = 
\left\{  
\begin{aligned}
&0 \text { if } \epsilon_1 =- \text{ and } \epsilon_2 = +\\
& [ Z_1^{\epsilon_1} Z_2^{\epsilon_2} ] \text{ if } \epsilon_1 \not=- \text{ or } \epsilon_2 \not= +
\end{aligned} 
\right.
$$
where $Z_1$ and $Z_2$ are the generators of $\TT$ associated to the sides $e_1$ and $e_2$ of $S$ indicated, and where $[ Z_1^{\epsilon_1} Z_2^{\epsilon_2} ]=\omega ^{-\epsilon_1\epsilon_2} Z_1^{\epsilon_1} Z_2^{\epsilon_2} =  \omega ^{\epsilon_1\epsilon_2} Z_2^{\epsilon_2} Z_1^{\epsilon_1}$ (identifying the sign $\epsilon = \pm$ to the exponent $\epsilon=\pm1$) is the Weyl quantum ordering defined in {\upshape \S \ref{sect:CheFock}};

\item in the case of Figure~{\upshape\ref{fig:TriangleOneArc}(b)}, where the end point of $K$ marked by $\epsilon_1$ is higher in $\partial S \times [0,1]$ than the point marked by $\epsilon_2$, 
$$
\Tr_\lambda^\omega \bigl([K, s]\bigr) = 
\left\{  
\begin{aligned}
0 &\text{ if } \epsilon_1 = \epsilon_2\\
 -\omega^{-5} &\text { if } \epsilon_1 =+ \text{ and } \epsilon_2 = -\\
 \omega^{-1} &\text{ if } \epsilon_1 =- \text{ and } \epsilon_2 = +
\end{aligned} 
\right.
$$
\end{enumerate}
\end{enumerate} 

In addition, the homomorphism $\Tr_\lambda^\omega \colon \SSS \to \TT$ is injective when the surface $S$ has no boundary. \qed
\end{thm}

\begin{figure}[htbp]
\SetLabels
(  .19 * -.2 )  (a) \\
( .81  * -.2 )  (b) \\
( .08 * .63 ) $\epsilon_1$ \\
( .3 * .63 ) $\epsilon_2$ \\
( .09 * .2 ) $e_1$ \\
( .3 * .2 ) $e_2$ \\
( .73 * .78 )  $\epsilon_1$\\
( .65 * .38 ) $\epsilon_2$ \\
\endSetLabels
\centerline{\AffixLabels{\includegraphics{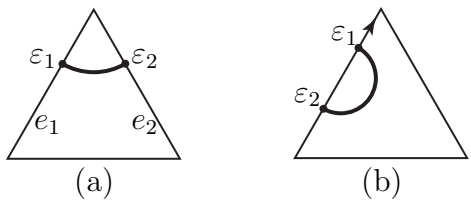}}}
\vskip 10pt
\caption{Elementary skeins in the triangle}
\label{fig:TriangleOneArc}
\end{figure}
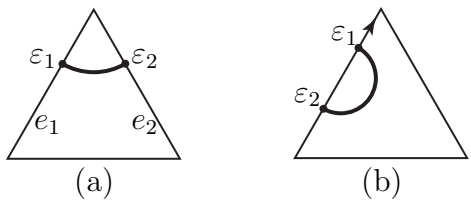

When $A=\pm1$ and $S$ has no boundary, $\Tr_\lambda^\omega\bigl([K]\bigr)$ is just the Laurent polynomial which, for an element $r \in \RS$ or $\RR$, expresses the trace $\Tr_r \bigl([K]\bigr) $ in terms of (the square roots of) the shear coordinates of $r$ with respect to $\lambda$. See \cite{BonWon1}.

When gluing skeins, it is crucial to keep track of the ordering of boundary points by their elevation. Figure~\ref{fig:TriangleOneArc}(b) illustrates a convenient method to describe this ordering in pictures representing a link $K$ in $S\times[0,1]$. For a component $d$ of $\partial S$, we first choose an orientation of $d$ and indicate it by an arrow. We then modify $K$ by an isotopy (respecting the elevations of boundary points) so that the ordering of the points of $\partial K \cap (d \times[0,1])$ by increasing elevation exactly corresponds to the ordering of their projections to $S$ for the chosen orientation of $d$; in this way, the ordering of these points becomes immediately clear on the picture. 
In addition, we will always assume in pictures that the framing is vertical, namely pointing towards the reader.

\begin{rem}
As is traditional in the physics literature, the states for a link $K$ are just a convenient way of specifying a basis element for a vector space $\bigotimes_{x\in \partial K}V_x$, where each $V_x$ is a 2--dimensional vector space endowed with a preferred basis $\{ v_+, v_-\}$. In this framework, the quantum trace map $\Tr_\lambda^\omega$ associates to each skein $[K]\in \SSS$  the  linear map $\bigotimes_{x\in \partial K}V_x \to \TT$ that, for each state $s\colon \partial K \to \{+, -\}$, assigns $\Tr_\lambda^\omega \bigl([K, s]\bigr)  \in \TT$ to the basis element $\bigotimes_{x\in \partial K}v_{s(x)} \in \bigotimes_{x\in \partial K}V_x$ corresponding to $s$. In this point of view, the State Sum Property just corresponds to a contraction property, using the basis $\{ v_+, v_-\}$ to identify the space $V_x$ to its dual $V_x^*$. 
\end{rem}

\subsection{Compatibility between the Chebyshev, quantum trace and Frobenius homomorphisms}

Theorem~\ref{thm:ChebyQTracesFrob} below shows that the  Chebyshev homomorphism $\mathbf T^A \colon \SSSS \to \SSS$, the Frobenius homomorphism $\mathbf F^A \colon \TTT \to \TT$, and the quantum trace homomorphisms $\Tr_\lambda^\omega \colon \SSS \to \TT$ and $\Tr_\lambda^\iota \colon \SSSS \to \TTT$ are remarkably compatible with each other. Theorem~\ref{thm:ChebFrobTraceIntro} in the introduction is an immediate consequence of this statement. 

\begin{thm}
\label{thm:ChebyQTracesFrob}
Let $S$ be a punctured surface with no boundary, and let $\lambda$ be an ideal triangulation of $S$. Then, if $A^4$ is a primitive $N$--root of unity, $A=\omega^{-2}$, $\epsilon = A^{N^2}$ and $\iota = \omega^{N^2}$, the diagram 
$$
\xymatrix{
\SSS
 \ar[r]^{\Tr_{\lambda}^\omega} 
 & \TT\\
\SSSS
 \ar[r]^{\Tr_{\lambda}^\iota}
 \ar[u]^{\mathbf T^A}
 & \TTT
 \ar[u]_{\mathbf F^\omega}}
$$
is commutative. Namely, for every skein $[K]\in \SSSS$, the quantum trace $\Tr_\lambda^\omega\bigl([K^{T_N}]\bigr)$ of $[K^{T_N}] = \mathbf T^A\bigl([K]\bigr)$ is obtained from the classical trace polynomial $\Tr_\lambda^\iota \bigl([K]\bigr)$ by replacing each generator $Z_i \in \TTT$ by $Z_i^N\in \TT$. 
\end{thm}

The proof of Theorem~\ref{thm:ChebyQTracesFrob} will occupy all of the next section \S\ref{sect:ChebyQTraces}. As should already be apparent from the statement, it involves a large number of ``miraculous cancellations''\kern -1pt.

\section{Quantum traces of Chebyshev skeins}
\label{sect:ChebyQTraces}

\subsection{Algebraic preliminaries}
\label{sect:Qnumbers}

For a parameter $a\in \C-\{0\}$, we will make use of two types of quantum integers,
$$
\QInt n{a} = \frac{a^{n}-1}{a-1} = a^{n-1} + a^{n-2} + a^{n-3} + \dots +1,
$$
and
$$
\QQInt{n}{a} = \frac{a^{n}-a^{-n}}{a-a^{-1}} = a^{n-1} + a^{n-3} + a^{n-5} + \dots + a^{-(n-1)} = a^{-(n-1)}\QInt{n}{a^2}.
$$
In particular, the first type of quantum integers will occur in the quantum factorials
$$
\QInt{n}{a}! = \QInt{n}{a} \QInt{n-1}{a} \dots \QInt{2}{a}\QInt{1}{a},
$$
and in the quantum binomial coefficients
$$
\QBinom{n}{p}{a}
=\frac
{\QInt{n}{a}\QInt{n-1}{a} \dots \QInt{n-p+2}{a}\QInt{n-p+1}{a}}
{\QInt{p}{a}\QInt{p-1}{a} \dots \QInt{2}{a}\QInt{1}{a}}
=\frac
{\QInt{n}{a}! }
{\QInt{p}{a}! \QInt{n-p}{a}! }.
$$
Although this definition of the quantum binomial coefficient  $\QBinom{n}{p}{a}$ requires that  $a^{k}\not=1$ for every $k\leq \min\{p, n-p\}$, see Lemmas~\ref{lem:Qbinom} or \ref{lem:QBinomInversions} below to make sense of it in all cases.

Quantum binomial coefficients naturally arise in the following classical formula (see for instance \cite[\S IV.2]{Kassel}). 

\begin{lem}[Quantum Binomial Formula]
\label{lem:Qbinom}
If the quantities  $X$ and $Y$ are such that $YX = a XY $, then 
$$
(X+Y)^n = \sum_{p=0}^n
 \QBinom{n}{p}{a} \kern -2pt
X^{n-p} Y^{p}.
$$
\vskip-\belowdisplayskip
\vskip -\baselineskip 
\vskip -7pt \qed

\vskip\belowdisplayskip
\vskip \baselineskip 
\end{lem}

We will encounter quantum binomial coefficients in a more combinatorial setting, where they occur as generating functions. 
Let $P$ be a linearly  ordered set of $n$ points. A \emph{state} for $P$ is a map $s\colon P \to \{ +, -\}$ assigning a sign $s(x)=+$ or $-$ to each point $x\in P$. 
An \emph{inversion} for $s$ is a pair $(x,x')$ of two points $x$, $x'\in P$ such that $x<x'$ for the ordering of $P$, but such that $s(x)>s(x')$ (namely $s(x) =+$ and $s(x')=-$). Let $\iota (s)$ denote the number of inversions of $s$, and let $|s|=\# \{ x\in P; s(x)=+\}$ be the number of $+$ signs in $s$. 

\begin{lem}
\label{lem:QBinomInversions}
$$
\QBinom{n}{p}{a}
=
\sum_{|s|=p} a^{\iota(s)}
$$
where the sum is over all states $s\colon \{1, 2, \dots, n\} \to \{ +, -\}$ with $|s|=p$. 
\end{lem}

\begin{proof}
This essentially  is a rephrasing of the Quantum Binomial Formula. Indeed, the expansion of $(X+Y)^n$ is the sum of all monomials $Z_1Z_2 \dots Z_n$ where each $Z_i$ is equal to $X$ or $Y$. Such  a monomial $Z_1Z_2 \dots Z_n$ can be described by a state $s\colon \{1, 2, \dots, n\} \to \{ +, -\}$, defined by the property that $s(i)=-$ if $Z_i=X$ and $s(i)=+$ when $Z_i=Y$. Then  $Z_1Z_2 \dots Z_n= a^{\iota(s)} X^{n-|s|}Y^{|s|}$ after reordering the terms. The result now follows from the Quantum Binomial Formula of Lemma~\ref{lem:Qbinom} by considering the coefficient of $X^{n-p}Y^p$. 
\end{proof}

\subsection{Quantum traces in the biangle}
In practice, computing the quantum trace of Theorem~\ref{thm:QTrace} is made easier by including  \emph{biangles} among allowable surfaces, namely infinite strips diffeomorphic to $[0,1]\times \R$. 

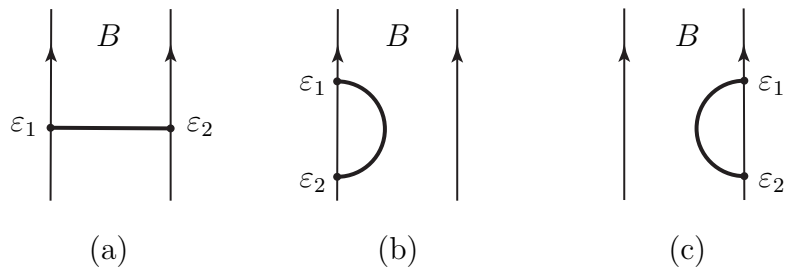
\begin{figure}[htbp]
\SetLabels
( -.03 * .36 ) $\epsilon_1$ \\
( .22 * .36 ) $\epsilon_2$  \\
( .38 * .57 ) $\epsilon_1$ \\
( .38 * .07 ) $\epsilon_2$  \\
( 1.03 * .57 ) $\epsilon_1$ \\
( 1.03 *  .07) $\epsilon_2$  \\
( .09 * -.3 )  (a) \\
( .5 * -.3 )  (b) \\
( .91 * -.3 ) (c)  \\
( .09 * .8 )  $B$ \\
( .5 * .8 )  $B$ \\
( .91 * .8 ) $B$  \\
\endSetLabels
\centerline{\AffixLabels{\includegraphics{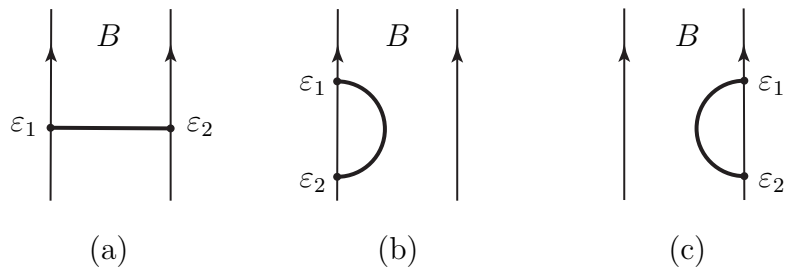}}}
\vskip 20 pt
\caption{Elementary skeins in the biangle}
\label{fig:BiangleOneArc}
\end{figure}

\begin{prop}[Proposition 13 and Lemma~14 of \cite{BonWon1}]
\label{prop:BiangleQTrace}
For $A=\omega^{-2}$, there  is a unique family of  algebra homomorphisms 
$$\Tr_B^\omega \colon \SSSSS(B) \to \C,$$
defined for all oriented biangles $B$, such that:  

\begin{enumerate}

\item {\upshape (State Sum Property)} if the biangle $B$ is obtained by gluing together two distinct biangles $B_1$ and $B_2$,  and if $[K_1] \in \mathcal S^A (B_1)$ and $[K_2]\in \mathcal S^A (B_2)$ are glued together to give $[K] \in  \mathcal S^A (B)$, then
$$
\quad \quad \quad
 \Tr_B^\omega \bigl([K,s]\bigr) = \kern -10pt \sum_{\text{compatible }s_1, s_2} \kern -10pt \Tr_{B_1}^\omega\bigl([K_1, s_1]\bigr) \Tr_{B_2}^\omega\bigl([K_2, s_2]\bigr) ,$$
where the sum is over all states $s_1\colon \partial K_1 \to \{+,-\}$ and $s_2\colon \partial K_2 \to \{+,-\}$ that are compatible with $s\colon \partial K \to \{+,-\}$ and with each other; 

\item {\upshape (Elementary Cases)}  if  the biangle $B$ is represented by a vertical strip in the plane as in Figure~{\upshape \ref{fig:BiangleOneArc}} and if $K$  projects to a single arc embedded in $B$, then
\begin{enumerate}
\item  in the case of Figure~{\upshape \ref{fig:BiangleOneArc}(a)}, where $\epsilon_1$, $\epsilon_2=\pm$ are the signs associated by the state $s$ to the end points of $K$,
$$
\Tr_B^\omega\bigl([K, s]\bigr) = 
\left\{  
\begin{aligned}
&1 \text { if } \epsilon_1 = \epsilon_2 \\
&0 \text{ if } \epsilon_1 \not= \epsilon_2 ;
\end{aligned} 
\right. \quad \quad \quad
$$

\item in the case of Figure~{\upshape \ref{fig:BiangleOneArc}(b)}, 
$$\quad \quad \quad \quad
\Tr_B^\omega\bigl([K, s]\bigr) = 
\left\{  
\begin{aligned}
0 &\text{ if } \epsilon_1 = \epsilon_2\\
 -\omega^{-5} &\text { if } \epsilon_1 =+ \text{ and } \epsilon_2 = -\\
\omega^{-1} &\text{ if } \epsilon_1 =- \text{ and } \epsilon_2 = + .
\end{aligned} 
\right.
$$

\item in the case of Figure~{\upshape \ref{fig:BiangleOneArc}(c)}, 
$$\quad \quad \quad \quad
\Tr_B^\omega\bigl([K, s]\bigr) = 
\left\{  
\begin{aligned}
0 &\text{ if } \epsilon_1 = \epsilon_2\\
 \omega &\text { if } \epsilon_1 =+ \text{ and } \epsilon_2 = -\\
-\omega^{5} &\text{ if } \epsilon_1 =- \text{ and } \epsilon_2 = + .
\end{aligned} 
\right.
$$
\vskip -\baselineskip
\vskip -\belowdisplayskip
\qed
\end{enumerate}

\end{enumerate}
\end{prop}

Note that Figure~\ref{fig:BiangleOneArc} uses the picture conventions indicated at the end of  {\upshape \S\ref{sect:QTrace}}. In particular, framings are everywhere vertical and, in Figures~\ref{fig:BiangleOneArc}(b) and (c), the end point labelled by $\epsilon_1$ sits at a higher elevation than the point labelled by $\epsilon_2$. 

The quantum trace $\Tr_B^\omega$ of Proposition~\ref{prop:BiangleQTrace} is just a version of the classical Kauffman bracket for tangles.

The relationship between the quantum traces of Proposition~\ref{prop:BiangleQTrace} and Theorem~\ref{thm:QTrace} is the following.  If we glue one side of a biangle $B_1$ to one side of a triangle $T_2$, the resulting surface is a triangle $T$. Each side of $T$ is naturally associated to a side of $T_2$, so that there is a natural isomorphism $ \mathcal T^\omega(T) \cong \mathcal T^\omega(T_2)$ between their associated Chekhov-Fock algebras. 

The following is an immediate consequence of the construction of the quantum trace homomorphism  in \cite[\S 6]{BonWon1} (which uses biangles as well as triangles). 

\begin{prop}
\label{prop:BiangleTriangleStateSum}
 If the triangle $T$ is obtained by gluing a biangle $B_1$ to a triangle $T_2$,  and if $[K_1] \in  \mathcal S^A (B_1)$ and $[K_2]\in \mathcal S^A (T_2)$ are glued together to give $[K] \in  \mathcal S^A (T)$, then for $A=\omega^{-2}$
$$
 \Tr_T^\omega \bigl([K,s]\bigr) = \kern -10pt\sum_{\text{compatible }s_1, s_2}\kern -10pt \Tr_{B_1}^\omega\bigl([K_1, s_1]\bigr) \Tr_{T_2}^\omega\bigl([K_2, s_2]\bigr) \in  \mathcal T^\omega(T_2) \cong  \mathcal T^\omega(T),$$
where the sum is over all states $s_1\colon \partial K_1 \to \{+,-\}$ and $s_2\colon \partial K_2 \to \{+,-\}$ that are compatible with $s\colon \partial K \to \{+,-\}$ and with each other.  \qed
\end{prop}

Note that, because a triangle $T$ has a unique triangulation $\lambda$, we write  $\Tr_T^\omega$ and $ \mathcal T^\omega(T)$ instead of  $\Tr_\lambda^\omega$ and $ \mathcal T^\omega(\lambda)$, by a slight abuse of notation. 

\subsection{Jones-Wenzl idempotents}

The \emph{Jones-Wenzl idempotents} are special elements  $JW_n\in\mathcal S^A(B)$ of the skein algebra of the biangle $B=[0,1]\times \R$. The $n$--th Jones Wenzl idempotent $JW_n$ is a linear combination of skeins, each of which meets each of the two components of $\partial B \times [0,1]$ in $n$ points. 

It is convenient (and traditional) to represent $JW_n$ by a box with $n$ strands emerging on each side, as on the left-hand side of Figure~\ref{fig:JWRecRel}. In this case, it is important to remember that the box does not represent a single link, but a linear combination of links. When several boxes appear on a picture, multiple sums need to be taken. 

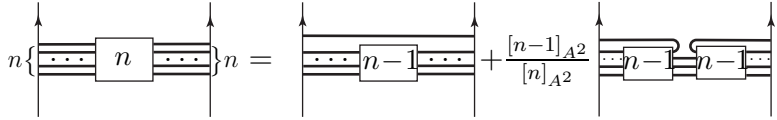
\begin{figure}[htbp]

\SetLabels
\E( .13 * .5) \small $n $ \\
(.48 *.41 ) \small$n\kern -2pt -\kern -2pt 1 $ \\
( .819*.41 ) \small $n\kern -2pt -\kern -3pt 1 $ \\
(.915 *.41 ) \small $n\kern -2pt -\kern -3pt 1 $ \\
\E( 0*.48 ) \scalebox{.8}{$n\kern -2pt\left\{ \raisebox{0pt}[10pt]{}\right. $} \\
\E( .26*.48 )\scalebox{.8}{$\left. \raisebox{0pt}[10pt]{}\right\}\kern -2pt n $}  \\
\E( .06 * .495) $\dots$\\
\E( .21 * .495) $\dots$\\
\E( .405 * .495) $\dots$\\
\E( .555 * .495) $\dots$\\
\E( .773 * .495) \scalebox{.6}{$\dots$}\\
\E( .966 * .495) \scalebox{.6}{$\dots$}\\
(.31 * .42) $=$\\
(.67 * .42) $+ \frac{\QQInt{n-1}{A^2}}{\QQInt{n}{A^2}}$\\
\endSetLabels
\centerline{\AffixLabels{ \includegraphics{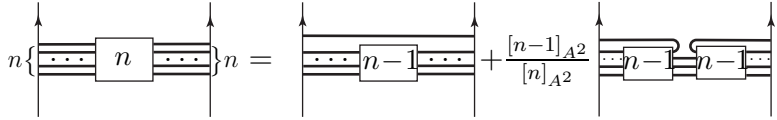} }}

\caption{The Jones-Wenzl recurrence relation}
\label{fig:JWRecRel}
\end{figure}

The Jones-Wenzl idempotents are then inductively defined by the property that $JW_1$ is the skein consisting of a single unknotted arc with vertical framing, as in Figure~\ref{fig:BiangleOneArc}(a), and by the recurrence relation indicated in Figure~\ref{fig:JWRecRel}. In particular, $JW_n$ is defined as long as $\QQInt{k}{A^2} \not=0$, or equivalently $A^{4k}\not=1$, for every $k\leq n$. 

The book chapter \cite[Chap.~13]{Lick} or the survey article \cite{Lick2} are convenient references for Jones-Wenzl idempotents. We will use the relations in $\mathcal S^A(B)$ summarized in Figures~\ref{fig:JWidemProp}, \ref{fig:UturnJW} and \ref{fig:JWloop}.

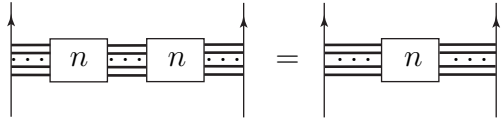
\begin{figure}[htbp]

\SetLabels
\E(.16 *.5 ) $ n$ \\
\E( .35* .5) $ n$ \\
\E( .81*.5 ) $ n$ \\
\E(.565 * .48) $= $ \\
\E( .065*.495) $\dots$\\
\E( .255*.495) $\dots$\\
\E( .445*.495) $\dots$\\
\E( .7*.495) $\dots$\\
\E( .925*.495) $\dots$\\
\endSetLabels
\centerline{\AffixLabels{ \includegraphics{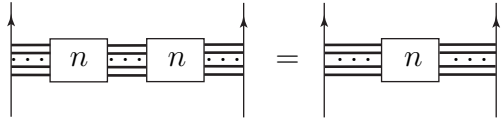} }}

\caption{The Jones-Wenzl idempotent property}
\label{fig:JWidemProp}
\end{figure}

\begin{figure}[htbp]

\SetLabels
\E(.145 * .5) $ n$ \\
\E (.39 * .5) $ =$  \\
( .23*.635 ) $ \dots$ \\
(.23 *.33 ) $\dots $ \\
( .065*.58 ) $ \dots$ \\
(.065 *.38 ) $\dots $ \\
\E(.63 * .5) $ n$ \\
( .556*.632 ) $ \dots$ \\
(.556 *.328 ) $\dots $ \\
( .72*.58 ) $ \dots$ \\
(.72 *.38 ) $\dots $ \\
\E (.85*.5) $=\kern 5pt0$\\
\endSetLabels
\centerline{\AffixLabels{ \includegraphics{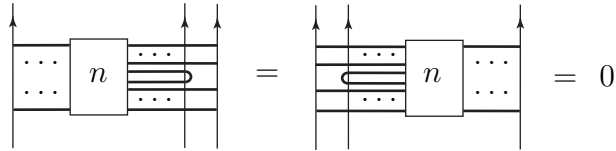} \hskip 2cm}}
\caption{Gluing a U-turn to a Jones-Wenzl idempotent}
\label{fig:UturnJW}

\end{figure}

\begin{figure}[htbp]

\SetLabels
\E( .18* .5) $n $ \\
( .82*.41 ) $ n\kern -2pt - \kern -2pt1$ \\
\E( .5*.5 ) $\quad = -\frac {\QQInt{n+1}{A^2}}{\QQInt{n}{A^2}}$ \\
( .085*.48 ) $\dots $ \\
( .29*.48 ) $\dots $ \\
( .72*.487 ) $\dots $ \\
( .93*.487 ) $\dots $ \\
\endSetLabels
\centerline{\AffixLabels{ \includegraphics{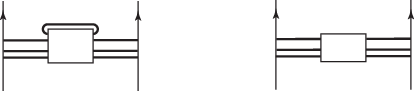} }}

\caption{}
\label{fig:JWloop}
\end{figure}

Note that the Jones-Wenzl idempotents $JW_n\in\mathcal S^A(B)$ are idempotent for the gluing property of Figure~\ref{fig:JWidemProp}, and not for the multiplication of $\mathcal S^A(B)$. See \cite[\S3.5]{CFS} for a more conceptual interpretation of this idempotent property, in terms of irreducible components of tensor products of the fundamental representation of the quantum group $\mathrm U_q(\mathfrak{sl}_2)$.

What makes Jones-Wenzl idempotents convenient for us is the following property. If we glue the two boundary components of the biangle $B$ together, we obtain a cylinder $C$.

\begin{lem}
\label{lem:ChebJW2}
Let the cylinder $C\cong S^1 \times \R$ be obtained by gluing together the two sides of the biangle $B=[0,1]\times \R$. Then, the element of $\mathcal S^A(C)$ obtained by applying this gluing to the Jones-Wenzl idempotent $JW_n \in \mathcal S^A(B)$ is equal to the evaluation $S_n\bigl([K]\bigr)$ of the Chebyshev polynomial $S_n$ at the skein  $[K]\in \mathcal S^A(C)$ represented by an embedded loop $K\subset S\times\{\frac12\}$ going around the cylinder, endowed with the vertical framing. 
\end{lem}

\begin{proof}
See for instance \cite[\S 13]{Lick} or \cite[p.~715]{Lick2}. 
\end{proof} 

We will use Lemma~\ref{lem:ChebJW2} to compute the quantum traces $\Tr_\lambda^\omega \bigl( [K^{S_N}]\bigr)$ and $ \Tr_\lambda^\omega \bigl( [K^{T_N}]\bigr)$ of the elements $[K^{S_N}]$, $[K^{T_N}] \in \SSS$ respectively obtained by threading the Chebyshev polynomials $S_N$  and $T_N$ along a framed link $K$. A particularly useful feature is the idempotent property of Figure~\ref{fig:JWidemProp}, which will enable us to localize the computations. 

Jones-Wenzl idempotents however have the drawback that not all of them are defined when $A$ is a root of unity (and in particular that $JW_N$ is undefined when $A^{4N}=1$, which is precisely that case we are interested in). Consequently, we will temporary assume that $A$ is \emph{generic}, namely is not a root of unity, to ensure that all Jones-Wenzl idempotents are defined. We will then let $A$ tend to an appropriate root of unity, and analyze what happens in the limit to complete our computation.

\subsection{Quantum traces of Jones-Wenzl idempotents in the biangle}

As indicated above, we assume for a while that $A$ is generic. 
Let us fix an integer $N\geq 1$.

Let $s$ be a state for the Jones-Wenzl idempotent $JW_N \in \mathcal S^A(B)$. This means that we are assigning the same state $s$ to each of the links that appear in the expression of $JW_N$. It is here important that all these links have the same  boundary, which we denote by $\partial JW_N$. 

Let $e_1$ and $e_2$ be the two boundary components of the biangle $B$, oriented in a parallel way and so that $e_1$ is the component on the left. 
The state $s$ for $JW_N$ consists of a state $s_1$ on $e_1\cap \partial JW_N$ and a state $s_2$ on $e_2\cap \partial JW_N$ (where $e_i\cap \partial JW_N$ is  shorthand notation for $(e_1\times [0,1])\cap \partial JW_N$). By convention, we order the  points of $e_i\cap \partial JW_N$ from bottom to top, in increasing order vertically upwards along the $[0,1]$--direction. As in \S \ref{sect:Qnumbers}, let  $|s_i|$ and $\iota(s_i)$ respectively denote the number of $+$ signs and the number of inversions of $s_i$. 

\begin{prop}
\label{prop:QTraceBiangleJW}
For a generic $A\in \C-\{0\}$, 
\begin{align*}
\Tr_B(JW_N, s) =0 \phantom{\frac{A^{2\iota(s_1)}A^{2\iota(s_2)}}{}} \quad&\text{if } |s_1| \not= |s_2|\text{, and}\\
\Tr_B(JW_N, s) = \frac
{A^{2\iota(s_1)}A^{2\iota(s_2)}}
{\QBinom{N}{p}{A^4}} \phantom{0}\quad
&\text{if } |s_1|=|s_2|=p.
\end{align*}
\end{prop}

\begin{proof} The first property is an easy consequence of the definition of the quantum trace of Proposition~\ref{prop:BiangleQTrace}; see \cite[Lemma~21]{BonWon1}. 

For the second property, we first reduce our computation to the case where there is no inversion.

Suppose that $s_1$ has at least one inversion, corresponding to two points $x$, $x'\in e_1\cap \partial JW_N$ such that $x<x'$ for the ordering of $e_1\cap \partial JW_N$,  $s(x)=+$ and $s(x')=-$. Without loss of generality, we can assume that $x$ and $x'$ are next to each other for the order on  $e_1\cap \partial JW_N$.  Let the state $s'$ be obtained from $s$ by exchanging $s(x)$ and $s(x')$. In particular, $\iota(s'_1) = \iota(s_1)-1$. 

Glue to $B$ another biangle $B_1$  along $e_1$, and glue to $[JW_N]\in \mathcal S^\omega(B)$ along the points $x$ and $x'$ a U-turn $K_1$ as in Figure~\ref{fig:UturnJW}. By the identity described by Figure~\ref{fig:UturnJW},  the resulting element of $ \mathcal S^A (B\cup B_1)$ is trivial. Combining the State Sum Property of Proposition~\ref{prop:BiangleQTrace} with the values of the quantum trace for the skein of Figure~\ref{fig:BiangleOneArc}(c), 
$$0 = \omega \,\Tr_B(JW_N, s') - \omega^5 \,\Tr_B(JW_N, s) $$
and therefore 
$
\Tr_B(JW_N, s)= A^2 \,\Tr_B(JW_N, s')
$
since $A=\omega^{-2}$. 

The same formula holds when $s'$ is obtained from $s$ by removing an inversion of consecutive points in $e_2$. This proves that 
$$
\Tr_B(JW_N, s) = {A^{2\iota(s_1)}A^{2\iota(s_2)}}\Tr_B(JW_N, s_0)
$$
where $s_0$ is the state that has no inversion, and has $p$ signs $+$ on $e_1\cap \partial JW_N$ and on $e_2\cap \partial JW_N$.

It therefore suffices to show that  $C(N,p) = \Tr_B(JW_N, s_0)$ is equal to $\QBinom{N}{p}{A^4}^{-1}$. 
We will prove this by induction on the number  $N-p$ of minus signs in the states $s_1$ and $s_2$.

Using the recurrence relation of Figure~\ref{fig:JWRecRel} to expand  $[JW_N]\in \mathcal S^A(B)$ as a linear combination of skeins, the only skein in this expansion that does not contain a U-turn as in Figure~\ref{fig:BiangleOneArc}(b) is the one consisting of $N$ parallel strands as in Figure~\ref{fig:BiangleOneArc}(a), and the coefficient of this term in the expansion is equal to 1. It follows that $C(N,N)=1$, which proves the initial step of the induction, when $N-p=0$.

For the general case, suppose the property proved for every $N'$, $p'$ with $N'-p' < N-p$. 
Consider the identity in $\mathcal S^A(B)$ represented in Figure~\ref{fig:JWloop} (see \cite[p. 137]{Lick} or \cite[p.~714]{Lick2}  for a proof). Endow both sides with the state that has no inversion, and had $p$ signs $+$ on each side of $B$. If we use the State Sum Property to compute the quantum trace of the left hand side of the equation of Figure~\ref{fig:JWloop}, only two terms have non-trivial contributions,  and this relation  gives
$$
\QQInt{N+1}{A^2} C(N-1, p) = - \QQInt{N}{A^2} \left( (-\omega^5) C(N,p+1) \omega^{-1}  + \omega A^{2p}A^{2p} C(N,p)(-\omega^{-5})
\right).
$$

Remembering that $\QQInt{N}{A^2} = A^{-2(N-1)} \QInt{N}{A^4}$ and $A=\omega^{-2}$, and using the induction hypothesis,
\begin{align*}
C(N,p) &= A^{-4(p+1)} \frac{\QInt{N+1}{A^4}}{\QInt{N}{A^4}} C(N-1, p)
- A^{-4(p+1)} C(N, p+1) \\
&= A^{-4(p+1)} \frac{\QInt{N+1}{A^4}}{\QInt{N}{A^4}} 
\frac{\QInt{p}{A^4} \QInt{p-1}{A^4} \dots \QInt{1}{A^4}}
{\QInt{N-1}{A^4} \QInt{N-2}{A^4} \dots \QInt{N-p}{A^4}}\\
&\quad\quad\quad \quad \quad \quad \quad \quad\quad \quad \quad - A^{-4(p+1)}
\frac{\QInt{p+1}{A^4} \QInt{p}{A^4} \dots \QInt{1}{A^4}}
{\QInt{N}{A^4} \QInt{N-1}{A^4} \dots \QInt{N-p}{A^4}}\\
&= A^{-4(p+1)}
\frac{\QInt{N+1}{A^4} - \QInt{p+1}{A^4}}{\QInt{N-p}{A^4}}
\frac{\QInt{p}{A^4} \QInt{p-1}{A^4} \dots \QInt{1}{A^4}}
{\QInt{N}{A^4} \QInt{N-1}{A^4} \dots \QInt{N-p+1}{A^4}} \\
&= \QBinom{N}{p}{A^4}^{-1} .
\end{align*} 
The last step uses the property that ${\QInt{N+1}{A^4} - \QInt{p+1}{A^4}}=A^{4(p+1)} {\QInt{N-p}{A^4}}$, which is an immediate consequence of the definition of the quantum integers $\QInt{k}{A^4}$. 

This concludes the proof by induction that $ C(N,p) = \QBinom{N}{p}{A^4}^{-1}$, and therefore that 
$$
\Tr_B(JW_N, s) = {A^{2\iota(s_1)}A^{2\iota(s_2)}}\Tr_B(JW_N, s_0)
={A^{2\iota(s_1)}A^{2\iota(s_2)}}\QBinom{N}{p}{A^4}^{-1}
$$
when $|s_1|=|s_2|=p$. \end{proof}

\subsection{Quantum traces of Jones-Wenzl skeins in the triangle}

We now consider a Jones-Wenzl idempotent  $JW_N$ in a triangle $T$, as represented in Figure~{\upshape\ref{fig:TriangleJW}(a)}.

\begin{figure}[htbp]

\SetLabels
( .2* .44) \tiny $N $ \\
( .72*.56 ) \rotatebox{-30}{\tiny$ N$} \\
( .2* -.2) (a) \\
( .8* -.2) (b) \\
( .04*.3) $e_1$\\
( .36*.3) $e_2$\\
( .6*.3) $e_1$\\
( .96*.3) $e_2$\\
( .75*.2) $e_0$\\
\endSetLabels
\centerline{\AffixLabels{ \includegraphics{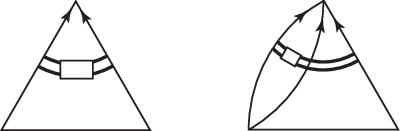}
 }}
\caption{}
\label{fig:TriangleJW}
\end{figure}

\begin{prop} 
\label{prop:QTraceTriangleJW}
In the triangle $T$, let $JW_N \in \mathcal S^A(T)$ be the Jones-Wenzl idempotent  represented in Figure~{\upshape\ref{fig:TriangleJW}(a)}. Let $e_1$ and $e_2$ be the two sides of $T$ indicated  in Figure~{\upshape\ref{fig:TriangleJW}(a)}, and let $Z_1$ and $Z_2$ denote the corresponding generators of $\mathcal T^\omega(T)$. Let $s$ be a state for $JW_N$, consisting of a state $s_1$ on $e_1\cap JW_N$ and a state $s_2$ on $e_2\cap JW_N$, let $\iota(s_1)$ and $\iota(s_2)$ be the respective numbers of inversions of $s_1$ and $s_2$, and set $p_1 = |s_1|$ and $p_2=|s_2|$. 

Then 
$\Tr_T^\omega\bigl([JW_N, s]\bigr)=0$ if $p_2>p_1$, and otherwise
\begin{align*}
\Tr_T^\omega\bigl([JW_N, s]\bigr) 
= 
{A^{2\iota(s_1)}  A^{2\iota(s_2)}} &
 \frac{\QInt{N-p_2}{A^4}! \QInt{p_1}{A^4}!}{\QInt{N}{A^4}! \QInt{p_1-p_2}{A^4}!}\\
  &\quad\quad\quad A^{-(p_1-p_2)(N-p_1+p_2)} [Z_1^{2p_1-N} Z_2^{2p_2-N}]
\end{align*}
where $[Z_1^{2p_1-N} Z_2^{2p_2-N}]$ denotes the Weyl quantum ordering for the monomial\linebreak $Z_1^{2p_1-N} Z_2^{2p_2-N}$, as defined in {\upshape\S \ref{sect:CheFock}}. 
\end{prop}

\begin{proof} The same U-turn trick as in the proof of Proposition~\ref{prop:QTraceBiangleJW} reduces  the computation to the case where there are no inversions. Therefore, we henceforth assume that $s_1$ and $s_2$ have no inversion.

Split the triangle $T$ into a biangle $B_1$ and a triangle $T_2$ as in Figure~\ref{fig:TriangleJW}(b), and let $e_0 $ be their common edge $ B_1 \cap T_2$. In particular, the Jones-Wenzl idempotent $JW_N$ in $T$ splits into a Jones-Wenzl idempotent $JW_N^1$ in the biangle $B_1$ and into a family $K_2$ of parallel strands in the triangle $T_2$. Applying the State Sum Property of Proposition~\ref{prop:BiangleTriangleStateSum},
$$
\Tr_{T}^\omega (JW_N, s) = \sum_{s_0} \Tr_{B_1}^\omega (JW_N, s_1 \cup s_0) \Tr_{T_2}^\omega (K_2, s_2 \cup s_0)
$$ 
where the sum is over all states $s_0$ for $e_0 \cap JW_N$. 

By Proposition~\ref{prop:QTraceBiangleJW}, a state $s_0$ with a non-trivial contribution to the above sum is such that $|s_0|=p_1$, and in this case
$$
 \Tr_{B_1}^\omega (JW_N, s_1 \cup s_0) = A^{2\iota(s_0)} 
 \QBinom{N}{p_1}{A^4}^{-1}.
$$

We now consider the terms coming from the triangle $T_2$. Let $s_j(i)\in \{ -, +\}$ denote the sign assigned by the state $s_j$ to the $i$--th point of $e_j \cap JW_n$. In particular, since $s_2$ has no inversion, $s_2(i)=+$ if and only if $i>N- p_2$. Therefore, by  Case~2(a) of Theorem~\ref{thm:QTrace}, if $\Tr_{T_2}^\omega (K_2, s_2 \cup s_0) \not=0$ then necessarily $s_0(i)=+$ for every $i>N- p_2$. In addition,   identifying  the sign $\pm$ to the number $\pm1$ in the exponents, this contribution is then equal to
$$
\Tr_{T_2}^\omega (K_2, s_2 \cup s_0) =
[Z_1^{s_0(1)}Z_2^{-1}] [Z_1^{s_0(2)}Z_2^{-1}] \dots [Z_1^{s_0(N-p_2)}Z_2^{-1}] [Z_1Z_2]^{p_2}.
$$
 Using the property that
$$
[Z_1Z_2^{-1}][Z_1^{-1}Z_2^{-1}] = \omega^{-4} [Z_1^{-1}Z_2^{-1}] [Z_1Z_2^{-1}] = A^{2} [Z_1^{-1}Z_2^{-1}] [Z_1Z_2^{-1}],
$$
 the terms in this contribution can be reordered as 
$$
\Tr_{T_2}^\omega (K_2, s_2 \cup s_0) = A^{2\iota(s_0)} 
 [Z_1^{-1}Z_2^{-1}]^{N-|s_0|} [Z_1Z_2^{-1}]^{|s_0|-p_2}  [Z_1Z_2]^{p_2}
$$

Combining this with Lemma~\ref{lem:QBinomInversions}, we obtain
\begin{align*}
\Tr_{T}^\omega (JW_N, s) &=  \kern -10 pt \sum_{\tiny
\begin{matrix}
s_0(i)=+ \text{ if } i>N- p_2\\
|s_0| = p_1
\end{matrix}
} \kern - 15pt
A^{4\iota(s_0)} 
 \QBinom{N}{p_1}{A^4}^{-1}
  [Z_1^{-1}Z_2^{-1}]^{N-p_1} [Z_1Z_2^{-1}]^{p_1-p_2}  [Z_1Z_2]^{p_2}\\
 &=  \QBinom{N-p_2}{p_1-p_2}{A^4} \QBinom{N}{p_1}{A^4}^{-1}
  [Z_1^{-1}Z_2^{-1}]^{N-p_1} [Z_1Z_2^{-1}]^{p_1-p_2}  [Z_1Z_2]^{p_2}\\
 &= \frac{\QInt{N-p_2}{A^4}! \QInt{p_1}{A^4}!}{\QInt{N}{A^4}! \QInt{p_1-p_2}{A^4}!}\,
   [Z_1^{-1}Z_2^{-1}]^{N-p_1} [Z_1Z_2^{-1}]^{p_1-p_2}  [Z_1Z_2]^{p_2}\\
  &= \frac{\QInt{N-p_2}{A^4}! \QInt{p_1}{A^4}!}{\QInt{N}{A^4}! \QInt{p_1-p_2}{A^4}!}\,
  A^{-(p_1-p_2)(N-p_1+p_2)} [Z_1^{2p_1-N} Z_2^{2p_2-N}]
\end{align*}
after a final grouping of terms. 

This concludes the proof of Proposition~\ref{prop:QTraceTriangleJW}. 
\end{proof}

\subsection{Evaluation of Chebyshev threads of the second kind}
\label{sect:QTraceChebyshev2}

Let $[K]\in \SSS$ be a skein in a surface $S$ without boundary, and let $\lambda$ be an ideal triangulation of $S$. We want to compute the image   $ \Tr^\omega _\lambda\bigl([K^{S_N }]\bigr) \in \TT$ of the element $[K^{S_N }] \in \SSS$  obtained by threading the Chebyshev polynomial of the second kind  $S_N$ along $K$.

We will restrict attention to the case where the skein $[K]\in \SSS$ is \emph{simple}, in the sense that it is represented by a framed knot $K \subset S \times [0,1]$ whose projection to $S$ is a simple closed curve and whose framing is vertical. As we will see in \S \ref{subsect:ProofChebyQTracesFrob}, this is no big loss of generality as simple skeins generate the  algebra $\SSS$. 

 For such a simple skein $[K] \in \SSS$, arbitrarily pick an orientation for $K$. Let $e_{i_1}$, $e_{i_2}$, \dots, $e_{i_u}$, $e_{i_{u+1}}=e_{i_1}$ denote, in this order,  the edges of $\lambda$ that are crossed by the projection of $K$ to $S$. We can arrange by an isotopy that $e_{i_{k+1}} \neq e_{i_k}$ for every $k$. Let $T_{j_1}$, $T_{j_2}$, \dots $T_{j_u}$ be the triangles of $\lambda$ that are crossed by $K$, in such a way that $K$ crosses $T_{j_k}$ between $e_{i_k}$ and $e_{i_{k+1}}$. 
 
 In particular,  the edge $e_{i_k}$ determines two generators $Z_{i_k, j_{k-1}} \in \mathcal T^\omega(T_{j_{k-1}})$ and $Z_{i_k, j_k} \in \mathcal T^\omega(T_{j_k})$ in the Chekhov-Fock algebras of the adjacent triangles. If  we describe the surface $S$ as obtained by gluing the triangles $T_j$ together as in \S \ref{sect:QTrace}, the generator of $\TT \subset \bigotimes_{j=1}^m \mathcal T^\omega (T_j)$ associated to $e_{i_k}$  is then $Z_{i_k} = Z_{i_k, j_{k-1}}  \otimes Z_{i_k, j_{k}} $. 

When the projection of the knot $K$ to $S$ crosses $e_{i_k}$, the orientations of $K$ and $S$ determine a left and a right endpoint for $e_{i_k}$, and there are four possible configurations according to whether $e_{i_{k-1}}$ and $e_{i_{k+1}}$ are respectively  adjacent to the left or right endpoint of $e_{i_k}$ in the triangles $T_{j_{k-1}}$ and $T_{j_k}$. We will say that $K$ crosses $e_{i_k}$ in a \emph{left-left, left-right, right-left} or \emph{right-right pattern} accordingly.  For instance $K$ crosses $e_{i_k}$ in a left-right pattern if $e_{i_{k-1}}$ is adjacent to the left endpoint of $e_{i_k}$, and $e_{i_{k+1}}$ is adjacent to its right endpoint. 

Finally, the determination of the quantum traces $ \Tr^\omega _\lambda\bigl([K]\bigr)$ and $ \Tr^\omega _\lambda\bigl([K^{S_N }]\bigr)$ uses a careful control of the elevation of the strands of $K \subset S \times [0,1]$ above the faces $T_j$ of the triangulation $\lambda$, and often requires correction factors when the ordering of these elevations do not match near the edges $e_i$ of $\lambda$ (see \cite{BonWon1}). In order to simplify the computations, we require that $K$ meets the first edge  $e_{i_1}$ only once, which will enable us to sidetrack the correction terms. This is a strong requirement, but it will be sufficient for our purposes in \S\S \ref{sect:ChebPuncTorusSphere} and \ref{subsect:ProofChebyQTracesFrob}.

\begin{prop}
\label{prop:QTraceChebyshev2}
Let $[K]\in \SSS$ be a simple skein in the surface $S$, crossing the edges $e_{i_1}$, $e_{i_2}$, \dots, $e_{i_u}$ of the triangulation $\lambda$ in this order, crossing the face $T_{j_k}$ between $e_{i_k}$ and $e_{i_{k+1}}$, and crossing the first edge $e_{i_1}$ exactly once. Then, for every generic $A$, 
$$
 \Tr^\omega _\lambda\bigl([K^{S_N }]\bigr) = \sum_{p_1, \,p_2, \dots,\, p_u} a_0b_1 b_2 \dots b_u \,
\bigl \langle Z_{i_1}^{2p_1-N} Z_{i_2}^{2p_2-N} \dots Z_{i_u}^{2p_u -N} \bigr\rangle
$$
where the sum is over all integers $p_k$ with $0 \leq p_k \leq N$, where
$$
a_0 = \prod_{k=1}^u \frac{A^{(p_k - p_{k+1})^2}}{ \QInt{|p_k - p_{k+1}|}{A^4}!},
$$
where
 \begin{itemize}
\item[]
$
b_k = 
\begin{cases}
1 \text{ if } p_{k-1} \geq p_k \geq p_{k+1}\\
0 \text{ otherwise}
\end{cases}
$\kern -10pt
when $K$ crosses $e_k$ in a left-left pattern,
\item[]\medskip
$
b_k = 
\begin{cases}
1 \text{ if } p_{k-1} \leq p_k \leq p_{k+1}\\
0 \text{ otherwise}
\end{cases}
$\kern -10pt
when $K$ crosses $e_k$ in a right-right pattern,
\item[]\medskip
$
b_k = 
\begin{cases}
A^{2Np_k} \frac{\QInt{N-p_k}{A^4}!}{\QInt{p_k}{A^4}!}  \text{ if } p_{k-1} \geq p_k \leq p_{k+1}\\
0 \text{ otherwise}
\end{cases}
$\kern -10pt
when $K$ crosses $e_k$ in a left-right pattern, and
\item[]\medskip
$
b_k = 
\begin{cases}
A^{-2Np_k} \frac{\QInt{p_k}{A^4}!}{\QInt{N-p_k}{A^4}!} \text{ if } p_{k-1} \leq p_k \geq p_{k+1}\\
0 \text{ otherwise}
\end{cases}
$\kern -10pt
when $K$ crosses $e_k$ in a right-left pattern,
\end{itemize}
and where
\begin{multline*}
\bigl \langle Z_{i_1}^{2p_1-N} Z_{i_2}^{2p_2-N} \dots Z_{i_u}^{2p_u -N} \bigr\rangle = [Z_{i_1,j_1}^{2p_1-N}Z_{i_2,j_1}^{2p_2-N}] [Z_{i_2,j_2}^{2p_2-N}Z_{i_3,j_2}^{2p_3-N}]  \\
 \dots [Z_{i_{u-1},j_{u-1}}^{2p_{u-1}-N}Z_{i_u,j_{u-1}}^{2p_u-N}] [Z_{i_u,j_u}^{2p_u-N}Z_{i_1,j_u}^{2p_1-N}] .
\end{multline*}
\end{prop}
Note that the term $\bigl \langle Z_{i_1}^{2p_1-N} Z_{i_2}^{2p_2-N} \dots Z_{i_u}^{2p_u -N} \bigr\rangle$ is equal to the Weyl quantum ordering $\bigl [ Z_{i_1}^{2p_1-N} Z_{i_2}^{2p_2-N} \dots Z_{i_u}^{2p_u -N} \bigr]$ when the edges $e_{i_k}$ crossed by $K$ are all distinct, but  otherwise depends on our indexing of these edges. 

\begin{proof} 
We will use Jones-Wenzl idempotents. 

By Lemma~\ref{lem:ChebJW2}, the element $[K^{S_N}] \in \SSS$ can be obtained by threading the Jones-Wenzl idempotent $JW_N$ along $K$ (using a thickened annulus with core $K$ embedded in $S\times[0,1]$). Using the idempotent property of Figure~\ref{fig:JWidemProp}, we can even put a Jones-Wenzl idempotent in each subarc delimited by the intersection of $K$ with the edges of the ideal triangulation $\lambda$. Consequently, for each  $k$, replace the  subarc of $K$ that goes from $e_{i_k}$ to $e_{i_{k+1}}$ by a Jones-Wenzl idempotent $JW_N^{(k)}\in \mathcal T^\omega(T_{j_k})$ as in Figure~\ref{fig:TriangleJW}(a). Then  $[K^{S_N}]$ is equal to the element of $\SSS$ obtained by gluing the $JW_N^{(k)}$ together.

To apply  the State Sum Property of Theorem~\ref{thm:QTrace}, we arrange that $K$ steadily goes up in $S\times [0,1]$ in the $[0,1]$ factor as it traverses the triangles $T_{j_1}$, $T_{j_2}$, \dots, $T_{j_u}$, and then sharply goes down above a biangle neighborhood of $e_{i_1}$ to return to its  starting point. Because of our hypothesis that $K$ crosses $e_{i_1}$ only once, the gluing of $K\cap T_{j_1} \times [0,1]$ and $K\cap T_{j_u} \times [0,1]$ along $e_{i_1} \times [0,1]$ can be done without reshuffling the order of the strands of $K$ near $e_{i_1}$. In this way, we can avoid the correction factors above biangle neighborhoods of the edges $e_i$ that are usually required in the computation of the quantum trace of a general skein \cite{BonWon1}. 

The State Sum Property then gives that 
$$
\Tr^\omega _\lambda\bigl([K^{S_N }]\bigr) = \sum_{s_1,\, s_2, \dots, \,s_u} \Tr^\omega _{T_{j_1}} (JW_N^{(1)}, s_{1}\cup s_2) \Tr^\omega _{T_{j_2}} (JW_N^{(2)}, s_2\cup s_3) \dots \Tr^\omega _{T_{j_u}} (JW_N^{(u)}, s_{u}\cup s_1)
$$
where the $s_k$ range over all states for $JW_N^{(k-1)} \cap e_{i_k} \times [0,1] = JW_N^{(k)} \cap e_{i_k} \times [0,1]$. Letting $p_k = |s_k|$ be the number of $+$ signs in $s_k$, Proposition~\ref{prop:QTraceTriangleJW} computes the  contribution of each family of states $s_1$, $s_2$, \dots, $s_u$ as
$$
\prod_{k=1}^u \Tr^\omega _{T_{j_k}} ( JW_N^{(k)}, s_k\cup s_{k+1}) = c_1 \dots c_u \, A^{4\iota(s_1)}\dots A^{4\iota(s_u)} \,\bigl\langle Z_{i_1}^{2p_1-N} \dots Z_{i_u}^{2p_u-N}\bigr\rangle
$$
where
$$
c_k=
\begin{cases}
0 &\text{if } p_k<p_{k+1}\\
\frac{\QInt{N-p_{k+1}}{A^4}! \QInt{p_k}{A^4}!}{\QInt{N}{A^4}! \QInt{p_k-p_{k+1}}{A^4}!}\,
  A^{(p_k-p_{k+1})^2 +N(p_{k+1}- p_k)} &\text{if } p_k\geq p_{k+1}
\end{cases}
$$
when $e_k$ and $e_{k+1}$ are adjacent to the left, and
$$
c_k=
\begin{cases}
\frac{\QInt{N-p_k}{A^4}! \QInt{p_{k+1}}{A^4}!}{\QInt{N}{A^4}! \QInt{p_{k+1}-p_k}{A^4}!}\,
  A^{(p_k-p_{k+1})^2 +N(p_k-p_{k+1})} &\text{if } p_k\leq p_{k+1}\\
0 &\text{if } p_k>p_{k+1}
\end{cases}
$$
when $e_k$ and $e_{k+1}$  are adjacent to the right. 
If we fix the numbers $p_1$, $p_2$, \dots, $p_u$ and sum over all states $s_k$ with $|s_k|=p_k$, Lemma~\ref{lem:QBinomInversions} shows that 
$$
\sum_{|s_k|=p_k}  A^{4\iota(s_1)}  A^{4\iota(s_2)} \dots A^{4\iota(s_u)} = \QBinom{N}{p_1}{A^4}  \QBinom{N}{p_2}{A^4}   \dots \QBinom{N}{p_u}{A^4}  .
$$
Therefore, for a given set of numbers $p_1$, $p_2$, \dots, $p_u$, the contribution of the states $s_k$ with $|s_k|=p_k$ is equal to 
$$
 c_1 \dots c_u \, \QBinom{N}{p_1}{A^4}   \dots \QBinom{N}{p_u}{A^4}  \,\bigl\langle Z_1^{2p_1-N} \dots Z_k^{2p_k-N}\bigr\rangle.
$$
This product is often 0. When it is not, many of the quantum factorials involved in the coefficients $c_k$ and $\QBinom{N}{p_k}{A^4} = \frac{\QInt{N}{A^4}!}{\QInt{N-p_k}{A^4}!\QInt{p_k}{A^4}!}$ cancel out. For instance, all terms $\QInt{N}{A^4}!$ disappear. The remaining terms are then easily grouped as in the statement of Proposition~\ref{prop:QTraceChebyshev2}.
\end{proof}

\subsection{Evaluation of Chebyshev threads of the first kind}
\label{sect:QTraceChebyshev1}

We now turn to Chebyshev polynomials of the first kind $T_N$.  Remembering from Lemma~\ref{lem:FirstSecondChebyshevs} that $T_N = S_N -S_{N-2}$, we now want to evaluate
$$
\Tr^\omega _\lambda\bigl ([K^{T_N }] \bigr)  = \Tr^\omega _\lambda\bigl( [K^{S_N }] \bigr) - \Tr^\omega _\lambda \bigl( [K^{S_{N-2} }] \bigr)
$$
for a simple skein $[K]\in \SSS$ that satisfies the hypotheses of Proposition~\ref{prop:QTraceChebyshev2}. 

For this, it is convenient to rephrase the formula of Proposition~\ref{prop:QTraceChebyshev2} by putting more emphasis on the powers $n_k = 2p_k-N$ of the generators $Z_{i_k}$. Note that $-N \leq n_k \leq N$, and that $n_k$ has the same parity as $N$. 

We will say that a sequence $n_1$, $n_2$, \dots, $n_u$ is \emph{admissible} if each of  the corresponding $p_k = \frac{N+n_k}2$ contributes a non-trivial term to the formula of Proposition~\ref{prop:QTraceChebyshev2}, namely if $-N \leq n_k \leq N$ for every $k$, if each $n_k$ has the same parity as $N$, and if 
\begin{itemize}
\item[] $n_{k} \geq n_{k+1} $ when $e_{k}$ and $e_{k+1}$ are adjacent to the left, and 

\item[] $n_{k} \leq n_{k+1} $ when $e_k$ and $e_{k+1}$ are adjacent to the right.

\end{itemize}

Then Proposition~\ref{prop:QTraceChebyshev2} can be rephrased as

\begin{prop}
\label{prop:QTraceChebyshev2bis}
Under the hypotheses of Proposition~{\upshape\ref{prop:QTraceChebyshev2}}, 
$$
\Tr^\omega _\lambda \bigl ([K^{S_N }] \bigr) = \sum_{\text{admissible } n_1,\, n_2, \dots,\, n_u} a_0b_1 b_2 \dots b_u \,\bigl \langle Z_{i_1}^{n_1} Z_{i_2}^{n_2} \dots Z_{i_u}^{n_u} \bigr\rangle
$$
where
$$
a_0 = \prod_{k=1}^u \frac{A^{(\frac{n_k - n_{k+1}}2)^2}}{ \QInt{\frac{|n_k - n_{k+1}|}2}{A^4}!}
$$
where
\begin{itemize}
\item[] $b_k=1$ when $K$ crosses $e_k$ in a left-left or right-right pattern,

\item[] $b_k= A^{N(N+n_k)} \frac{\QInt{\frac{N-n_k}2}{A^4}!}{\QInt{\frac{N+n_k}2}{A^4}!} $ when $K$ crosses $e_k$ in a left-right pattern, and 

\item[] $b_k= A^{-N(N+n_k)} \frac{\QInt{\frac{N+n_k}2}{A^4}!}{\QInt{\frac{N-n_k}2}{A^4}!} $ when $K$ crosses $e_i$ in a right-left pattern
\end{itemize}
and where
\pushQED{\qed}
\begin{equation*}
\bigl \langle Z_{i_1}^{n_1} Z_{i_2}^{n_2} \dots Z_{i_u}^{n_u} \bigr\rangle = [Z_{i_1,j_1}^{n_1}Z_{i_2,j_1}^{n_2}] [Z_{i_2,j_2}^{n_2}Z_{i_3,j_2}^{n_3}]  
 \dots [Z_{i_{u-1},j_{u-1}}^{n_{u-1}}Z_{n_u, j_{u-1}}^{n_u}] [Z_{i_u,j_u}^{n_u}Z_{i_1,j_u}^{n_1}] . \qedhere
\end{equation*}
\end{prop}

An almost identical formula holds for $\Tr^\omega _\lambda\bigl([K^{S_{N-2} }]\bigr) $, except that the admissible sequences $n_1$, $n_2$, \dots, $n_u$ are further constrained by the condition that $-N+2 \leq n_k \leq N-2$ for every $k$. Because of the parity condition, this is equivalent to $-N< n_k < N$. 

More precisely,
$$
\Tr^\omega _\lambda \bigl ([K^{S_{N-2} }] \bigr) = \sum_{
\substack
{\text{admissible } n_1, \,n_2, \dots,\, n_u
\\ \text{with } -N<n_k<N}
} a_0b_1' b_2' \dots b_u' \,\bigl \langle Z_{i_1}^{n_1} Z_{i_2}^{n_2} \dots Z_{i_u}^{n_u} \bigr\rangle
$$
where $a_0$ is defined as in Proposition~\ref{prop:QTraceChebyshev2bis}, and
\begin{itemize}
\item[] $b_k'=1$ when $K$ crosses $e_{i_k}$ in a left-left or right-right pattern,

\item[] $b_k'= A^{(N-2)(N-2+n_k)} \frac{\QInt{\frac{N-2-n_k}2}{A^4}!}{\QInt{\frac{N-2+n_k}2}{A^4}!} $ when $K$ crosses $e_{i_k}$  in a left-right pattern,  

\item[] $b_k'= A^{-(N-2)(N-2+n_k)} \frac{\QInt{\frac{N-2+n_k}2}{A^4}!}{\QInt{\frac{N-2-n_k}2}{A^4}!} $ when $K$ crosses $e_{i_k}$  in a right-left pattern. 
\end{itemize}

So far, our computations assumed that $A$ was generic. We will see that many cancellations occur in this expression of $\Tr^\omega _\lambda \bigl ([K^{T_N }] \bigr)  = \Tr^\omega _\lambda \bigl ([K^{S_{N} }] \bigr)  - \Tr^\omega _\lambda \bigl ([K^{S_{N-2} }] \bigr) $ when $A^4$ is a primitive $N$--root of unity. However, because $\QInt{N}{A^4}=0$ in this case, we have to be careful in the definition of the quantities considered, and make sure that we never attempt to divide by 0. We will make sense of these properties by a limiting process.

\begin{lem}
\label{lem:ChebyshevDifference}
Let $n_1$, $n_2$, \dots, $n_u$ be an admissible sequence such that  $-N<n_k<N$ for every $k$. 
Then, the respective contributions
$$
a_0b_1 b_2 \dots b_u \,\bigl \langle Z_{i_1}^{n_1} Z_{i_2}^{n_2} \dots Z_{i_u}^{n_u} \bigr\rangle 
$$
and
$$
a_0b_1' b_2' \dots b_u' \,\bigl \langle Z_{i_1}^{n_1} Z_{i_2}^{n_2} \dots Z_{i_u}^{n_u} \bigr\rangle 
$$
of this admissible sequence to $\Tr^\omega _\lambda \bigl ([K^{S_N }] \bigr)$ and $\Tr^\omega _\lambda \bigl ([K^{S_{N-2} }] \bigr)$ have the same limit as $A^4$ tends to a primitive $N$--root of unity. 
\end{lem}

\begin{proof}
Because of the assumption that $-N<n_k<N$, all quantum integers involved in these contributions are different from $\QInt{N}{A^4}$ and we do not really have to worry about taking limits here. By continuity, choosing $A^4$ to be a primitive $N$--root of unity will be sufficient.  We therefore need to show that
$$
b_1 b_2 \dots b_u = b_1' b_2' \dots b_u'
$$
under this hypothesis on $A$. 

Let us compare the coefficients $b_k$ and $b_k'$. When $K$ crosses $e_{i_k}$ in a left-left or right-right patterns, we of course have that $b_k = b_k' =1$. 

When $K$ crosses $e_{i_k}$ in a left-right pattern,
\begin{align*}
b_k & = A^{N(N+n_k)} \frac{\QInt{\frac{N-n_k}2}{A^4}!}{\QInt{\frac{N+n_k}2}{A^4}!}
= 
A^{4N +2n_k-4}
 A^{(N-2)(N-2+n_k)}
  \frac{\QInt{\frac{N-n_k}2}{A^4}}{\QInt{\frac{N+n_k}2}{A^4}}
   \frac{\QInt{\frac{N-n_k}2-1}{A^4}!}{\QInt{\frac{N+n_k}2-1}{A^4}!}\\
   &= 
A^{4N +2n_k-4}
  \frac{\QInt{\frac{N-n_k}2}{A^4}}{\QInt{\frac{N+n_k}2}{A^4}}
  b_k'
  =A^{4N +2n_k-4}
  \frac{A^{2N-2n_k}-1}{A^{2N+2n_k}-1}
  b_k'
  =-A^{2N-4} b_k',
\end{align*}
using the fact that $A^{4N}=1$ for the last equality. 

When $K$ crosses $e_{i_k}$ in a right-left pattern, a similar computation gives
$$
b_k = -A^{-(2N-4)} b_k'.
$$

Note that $K$ crosses as many $e_{i_k}$ in a left-right pattern as in a right-left pattern. Therefore, as we compute the product of the $b_k$, the $-A^{\pm(2N-4)}$ terms cancel out, and 
\begin{equation*}
b_1 b_2 \dots b_u = b_1' b_2' \dots b_u' . \qedhere
\end{equation*}
\end{proof}

A consequence of Lemma~\ref{lem:ChebyshevDifference} is that, as we let $A^4$ tend to a primitive $N$--root of unity, all the terms of $ \Tr^\omega _\lambda \bigl( [K^{S_{N-2} }] \bigr)$ cancel out with terms of $\Tr^\omega _\lambda \bigl( [K^{S_N }] \bigr) $ in the difference $\Tr^\omega _\lambda \bigl( [K^{T_N }] \bigr) = \Tr^\omega _\lambda \bigl( [K^{S_N }] \bigr) - \Tr^\omega _\lambda \bigl( [K^{S_{N-2} }] \bigr)$. 

We now consider the remaining terms of $\Tr^\omega _\lambda \bigl( [K^{S_N }] \bigr)$.

\begin{lem}
\label{lem:ChebyshevCancel}
Let $n_1$, $n_2$, \dots, $n_u$ be an admissible sequence such that at least one $n_k$ is equal to $\pm N$ and at least one $n_l$ is \emph{not} equal to $\pm N$. 
Then, the contribution
$$
a_0b_1 b_2 \dots b_u \,\bigl \langle Z_{i_1}^{n_1} Z_{i_2}^{n_2} \dots Z_{i_u}^{n_u} \bigr\rangle 
$$
of this admissible sequence to $\Tr^\omega _\lambda \bigl ([K^{T_N }] \bigr)$ converges to $0$ as $A^4$ tends to a primitive $N$--root of unity. 
\end{lem}

\begin{proof}
 In the expression of $a_0b_1 b_2 \dots b_u$, the only quantum integers that can tend to $0$ are the terms $\QInt{N}{A^4}$. We therefore have to show that more quantum integers $\QInt{N}{A^4}$ occur in the numerator of this expression than in the denominator. 
 
 The number of terms $\QInt{N}{A^4}$  in the denominator of $a_0$ is equal to the number of indices $k$ where $n_k$ switches from $n_k=\pm N$ to $n_{k+1}=\mp N$. 

A factor $\QInt{N}{A^4}$ occurs in the numerator of a coefficient $b_k$ exactly when
\begin{itemize}
\item[] $n_k=-N$ and $K$ crosses $e_{i_k}$ in a left-right pattern, or
\item[] $n_k=+N$ and $K$ crosses $e_{i_k}$ in a right-left pattern.
\end{itemize}

Similarly, a factor $\QInt{N}{A^4}$ occurs in the denominator of a coefficient $b_k$ exactly when
\begin{itemize}
\item[] $n_k=+N$ and $K$ crosses $e_{i_k}$ in a left-right pattern, or
\item[] $n_k=-N$ and $K$ crosses $e_{i_k}$ in a right-left pattern.
\end{itemize}

Consider a maximal sequence of consecutive $n_k=+N$ namely, considering indices modulo $u$, two indices $k_1$, $k_2$ such that $n_k=+N$ whenever $k_1\leq k \leq k_2$, and $n_{k_1-1}\not=+N$ and $n_{k_2+1}\not=+N$. Note that, since the sequence of $n_k$ is admissible,  the edges $e_{i_{k_1-1}}$ and $e_{i_{k_1}}$ are necessarily adjacent on the right-hand side of $K$, whereas $e_{i_{k_2}}$ and $e_{i_{k_2+1}}$ are adjacent on the left. Therefore, if we examine how $K$ crosses $e_{i_k}$ when  $k_1\leq k \leq k_2$, we see one more right-left pattern than left-right patterns. It follows that this maximal sequence  of consecutive $n_k=+N$ contributes one more $\QInt{N}{A^4}$ to the numerator than to the denominator of the product of the corresponding $b_k$. 

Similarly, a maximal sequence of consecutive $n_k=-N$ contributes one more $\QInt{N}{A^4}$ to the numerator than to the denominator of the product of the corresponding $b_k$. 

Because of our assumption that there exists at least one $n_l \not = \pm N$, the number of $k$ where $n_k$ switches from $n_k=\pm N$ to $n_{k+1}=\mp N$ is strictly less that the total number of  maximal sequences of consecutive $n_k=+N$ and of maximal sequences of consecutive $n_k=-N$. It follows that there is at least one more $\QInt{N}{A^4}$ in the numerator of $a_0b_1 b_2 \dots b_u$ than in the denominator. This term therefore converges to 0 as $A^4$ tends to a primitive $N$--root of unity. 
\end{proof}

\begin{lem}
\label{lem:ChebyshevStay}
Let $n_1$, $n_2$, \dots, $n_u$ be an admissible sequence where each  $n_k$ is equal to $\pm N$. 
Then the contribution
$$
a_0b_1 b_2 \dots b_u \,\bigl \langle Z_{i_1}^{n_1} Z_{i_2}^{n_2} \dots Z_{i_u}^{n_u} \bigr\rangle 
$$
of this admissible sequence  to $\Tr^\omega _\lambda \bigl ([K^{T_N }] \bigr)$  is equal to $\bigl \langle Z_{i_1}^{n_1} Z_{i_2}^{n_2} \dots Z_{i_u}^{n_u} \bigr\rangle $. 
\end{lem}

\begin{proof}
The proof of Lemma~\ref{lem:ChebyshevCancel} shows that, in this case, we have exactly as many quantum factorials $\QInt{N}{A^4}!$ in the numerator as in the denominator of $a_0b_1 b_2 \dots b_u$. These quantum factorials therefore cancel out.

All remaining quantum factorials are equal to $\QInt{0}{A^4}!=1$. 

We therefore only have to worry about the powers of $A$ that occur in $a_0b_1 b_2 \dots b_u$. Going back to the definition of the constants $a_0$ and $b_k$ in Proposition~\ref{prop:QTraceChebyshev2bis}, one obtains that
$$
 a_0b_1 b_2 \dots b_u = A^{2N^2(\alpha+\beta-\beta')}
$$
where
\begin{align*}
\alpha &={\textstyle  \frac12 } \sum_{k=1}^u {\textstyle \left(\frac {n_k - n_{k+1}}{2N}\right)^2} = {\textstyle \frac12} { \# \{k; n_k \not= n_{k+1}\} }
\\
\beta&= \sum_{\text{left-right pattern at } e_{i_k}}  \kern -5pt {\textstyle \frac{N+n_k}N}
=\# \{k; n_k=+N \text{ and left-right pattern at } e_{i_k}\}
\\
\text{and } \beta' &= \sum_{\text{right-left pattern at } e_{i_k}}\kern -5pt {\textstyle  \frac{N+n_k}N}
=\# \{k; n_k=+N \text{ and right-left  pattern at } e_{i_k}\}.
\end{align*}

If all $n_k$ are equal to each other, then $\alpha =0$ and $\beta=\beta'$, so that  $\alpha+ \beta - \beta'=0$. 

Otherwise, $\alpha$ is equal to the number of intervals $I = \{k_1, k_1+1, \dots, k_2-1, k_2\}$  in the index set (counting indices modulo $u$) where $n_k=+N$ for every $k\in I$ while $n_{k_1-1}=n_{k_2+1}=-N$. For such an interval $I$, the indices $k\in I$ contribute a total of $-1$ to $\beta-\beta'$. Since there are $\alpha$ such intervals $I$, we conclude that $\alpha + \beta-\beta'=0$ in this case as well. 

This proves that $ a_0b_1 b_2 \dots b_u =1$ is all cases. \end{proof}

If we combine Proposition~\ref{prop:QTraceChebyshev2bis} and Lemmas~\ref{lem:ChebyshevDifference}, \ref{lem:ChebyshevCancel} and \ref{lem:ChebyshevStay}, we now have the following computation.

\begin{prop}
\label{prop:QTraceChebyshev1}
Suppose that $A^4$ is a primitive $N$--root of unity. Let $[K]\in \SSS$ be a simple skein in the surface $S$, crossing the edges $e_{i_1}$, $e_{i_2}$, \dots, $e_{i_u}$ of the triangulation $\lambda$ in this order, crossing the face $T_{j_k}$ between $e_{i_k}$ and $e_{i_{k+1}}$, and crossing the first edge $e_{i_1}$ exactly once. Then, 
$$
 \Tr^\omega _\lambda\bigl ([K^{T_N }] \bigr) = 
\sum_{\tiny
\begin{matrix}
\text{admissible } n_1,\, n_2, \dots,\, n_u\\
\text{with }n_k=\pm N
\end{matrix}
}
\bigl \langle Z_{i_1}^{n_1} Z_{i_2}^{n_2} \dots Z_{i_u}^{n_u} \bigr\rangle 
$$
where the sum is over all admissible sequences $n_1$, $n_2$, \dots, $n_k$ with all $n_k=\pm N$, and where
\pushQED{\qed}
\begin{equation*}
\bigl \langle Z_{i_1}^{n_1} Z_{i_2}^{n_2} \dots Z_{i_u}^{n_u} \bigr\rangle = [Z_{i_1,j_1}^{n_1}Z_{i_2,j_1}^{n_2}] [Z_{i_2,j_2}^{n_2}Z_{i_3,j_2}^{n_3}]  
 \dots [Z_{i_{u-1},j_{u-1}}^{n_{u-1}}Z_{n_u, j_{u-1}}^{n_u}] [Z_{i_u,j_u}^{n_u}Z_{i_1,j_u}^{n_1}] . \qedhere
\end{equation*}
\end{prop}

Note the dramatic difference between  the number of terms in the expression of $ \Tr^\omega _\lambda\bigl ([K^{T_N }] \bigr) $ provided by Proposition~\ref{prop:QTraceChebyshev1}, and that in the formula for generic $A$ given in  Proposition~\ref{prop:QTraceChebyshev2bis}. Indeed, the number of monomials in the formula of Proposition~\ref{prop:QTraceChebyshev1}  is independent of $N$. On the other hand, the number of terms in the expression of Proposition~\ref{prop:QTraceChebyshev2bis} is a polynomial in $N$ of degree $k$ (it is the Ehrhart polynomial of a certain $k$--dimensional polytope determined by the admissibility conditions).  

We can rephrase Proposition~\ref{prop:QTraceChebyshev1} in terms of the Frobenius homomorphism $\mathbf F^\omega \colon \TTT   \to \TT $ of Proposition~\ref{prop:Frobenius}. Recall that $\iota = \omega^{N^2}$. 
\begin{cor}
\label{cor:QTraceChebFrobSpecial}
Under the hypotheses of Proposition~{\upshape \ref{prop:QTraceChebyshev1}}, $ \Tr^\omega _\lambda\bigl ([K^{T_N }] \bigr) =\mathbf F^\omega \Bigl( \Tr^\iota _\lambda\bigl ([K] \bigr)\Bigr) $. 
\end{cor}
\begin{proof}
The conclusion of Proposition~\ref{prop:QTraceChebyshev1} can be rewritten as
$$
 \Tr^\omega _\lambda\bigl ([K^{T_N }] \bigr) = 
\sum_{\tiny
\begin{matrix}
\text{admissible } m_1, \,m_2, \dots, \,m_u\\
\text{with }m_k=\pm 1
\end{matrix}
}
\bigl \langle Z_{i_1}^{m_1N} Z_{i_2}^{m_2N} \dots Z_{i_u}^{m_uN} \bigr\rangle .
$$
Also, by definition of the quantum trace homomorphism in Theorem~\ref{thm:QTrace} (or replacing $N$ by 1 and $\omega$ by $\iota$ in Proposition~\ref{prop:QTraceChebyshev1}),
$$
 \Tr^\iota _\lambda\bigl ([K] \bigr) = 
\sum_{\tiny
\begin{matrix}
\text{admissible } m_1,\, m_2, \dots,\, m_u\\
\text{with }m_k=\pm 1
\end{matrix}
}
\bigl \langle Z_{i_1}^{m_1} Z_{i_2}^{m_2} \dots Z_{i_u}^{m_u} \bigr\rangle .
$$
The result then immediately follows from the fact that $\mathbf F^\omega$ is induced by the homomorphism $\bigotimes_{j=1}^m \mathcal T^\iota(T_j) \to \bigotimes_{j=1}^m \mathcal T^\omega(T_j)$ that sends each $Z_{i,j} \in \mathcal T^\iota(T_j)$ to $Z_{i,j}^N \in \mathcal T^\omega(T_j)$. 
\end{proof}

\subsection{Proof of Theorems~\ref{thm:ChebyshevCentral} and \ref{thm:ChebSkeinRelation}} 
\label{sect:ChebPuncTorusSphere} 

In \S \ref{sect:ThreadCheb} and \S\ref{sect:ChebHom}, we had not finished proving Theorems~\ref{thm:ChebyshevCentral} and \ref{thm:ChebSkeinRelation}, namely the statements that Chebyshev threads $[K^{T_N}]$ are central in $\SSS$ and that the Chebyshev map $\mathbf T^A \colon \SSSS \to \SSS$ is a well-defined algebra homomorphism. Indeed, our arguments relied on three lemmas whose proofs we had temporarily postponed,  Lemmas~\ref{lem:CentralPuncturedTorus}, \ref{lem:ChebSkeinRelationPuncTorus} and \ref{lem:ChebSkeinRelationPuncSphere}. This section is devoted to proving these statements, using the  computations of the previous section.

\begin{figure}[htbp]
\SetLabels
(.05 * .55) $L_0 $ \\
( .34*  .7) $L_\infty $ \\
(.6 *  .64) $ L_1$ \\
(.83 *  .46) $L_{-1} $ \\
( .1* 1) $e_0 $ \\
( .08* .2) $ e_1$ \\
(-.02 * .4) $ e_\infty$ \\
( .07* .75) $ T_1$ \\
( .17* .22) $T_2 $ \\
\endSetLabels
\centerline{\AffixLabels{\includegraphics{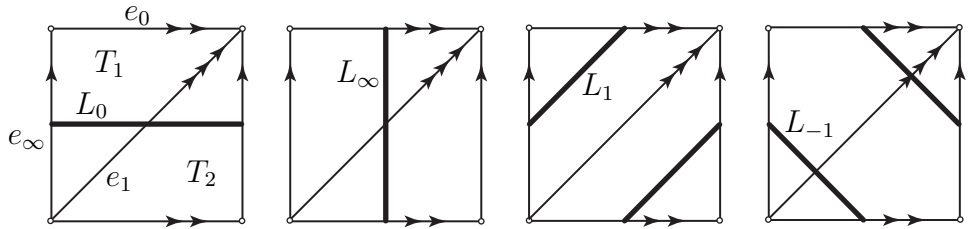}}}

\caption{The curves $L_0$, $L_\infty$, $L_1$ and $L_{-1}$ in the once-punctured torus, with an ideal triangulation $\lambda$}
\label{fig:PuncturedTorus2}
\end{figure}
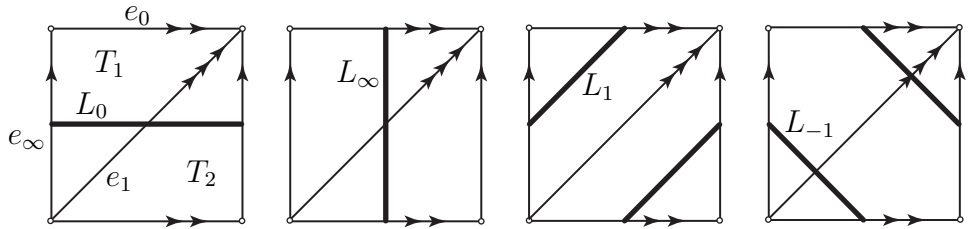

We begin with Lemma~\ref{lem:CentralPuncturedTorus}, which we repeat for the convenience of the reader.

\begin{lem}
\label{lem:CentralPuncturedTorus2}
In the once-punctured torus $T$, let $L_0$ and $L_\infty$ be the two curves represented in Figure~{\upshape\ref{fig:PuncturedTorus2}} (or Figure~{\upshape\ref{fig:PuncturedTorus}}), and consider these curves as framed knots with vertical framing in $T\times[0,1]$. If $A^2$ is a primitive $N$--root of unity, then
$$
[L_0^{T_N}] [L_\infty] = [L_\infty][L_0^{T_N}] 
$$
in $\mathcal S^A(T)$. 
\end{lem}

\begin{proof}
Consider the ideal triangulation $\lambda$ represented in Figure~\ref{fig:PuncturedTorus2}, with edges $e_0$, $e_1$ and $e_\infty$.  
The quantum trace map $\Tr_\lambda^\omega \colon \mathcal S^A(T) \to \TT$ is an injective algebra homomorphism \cite[Prop.~29]{BonWon1}, so that it suffices to check that
$$
\Tr_\lambda^\omega\bigl([L_0^{T_N}] \bigr) \Tr_\lambda^\omega\bigl([L_\infty] \bigr)= \Tr_\lambda^\omega\bigl( [L_\infty] \bigr) \Tr_\lambda^\omega\bigl(  [L_0^{T_N}] \bigr) .
$$

First consider the case where $N$ is odd. Then, $A^4$ is also a primitive $N$--root of unity, and we can use the computations of \S \ref{sect:QTraceChebyshev1}.
Let $Z_0$, $Z_1$, $Z_\infty$ be the generators of $\TT$ respectively associated to the edges $e_0$, $e_1$, $e_\infty$ of $\lambda$. Proposition~\ref{prop:QTraceChebyshev1} then shows that the quantum trace $\Tr_\lambda^\omega\bigl([L_0^{T_N}] \bigr)$ is a Laurent polynomial in the variables $Z_1^N$ and $Z_\infty^N$. 

Note that, when $i\not=j$, the two ends of the edge $e_i$ are adjacent to the two ends of $e_j$, so that $Z_iZ_j = \omega^{\pm4} Z_jZ_i$. In particular, $Z_i^N$ commutes with $Z_j$ since $\omega^{4N}=A^{-2N}=1$.  As a consequence,  $\Tr_\lambda^\omega\bigl([L_0^{T_N}] \bigr)$ is central in $\TT$, and in particular commutes with $\Tr_\lambda^\omega\bigl([L_\infty] \bigr)$.  This proves the desired property when $N$ is odd. 

When $N$ is even, $A^4$ is a primitive $\frac N2$--root of unity. Proposition~\ref{prop:QTraceChebyshev1}  now shows that $\Tr_\lambda^\omega\bigl([L_0^{T_{N/2}}] \bigr) = T_{\frac N2} \bigl(\Tr_\lambda^\omega ([L_0] )\bigr) $ is a linear combination of monomials $Z_1^{n_1}Z_\infty^{n_2}$ with $n_1$, $n_2\in \{-\frac N2, +\frac N2\}$. By definition of Chebyshev polynomials, $T_N = T_2\circ T_{\frac N2}$ (use for instance Lemma~\ref{lem:Chebyshev}) and $T_2(x) = x^2-2$. Therefore
$$
\Tr_\lambda^\omega\bigl([L_0^{T_{N}}] \bigr) 
=T_{N} \bigl(\Tr_\lambda^\omega ([L_0] )\bigr)
=T_2 \left( T_{\frac N2} \bigl( \Tr_\lambda^\omega([L_0] )\bigr)  \right)
= \left( \Tr_\lambda^\omega\bigl([L_0^{T_{N/2}}] \bigr) \right)^2 -2
$$
is a linear combination of monomials $Z_1^{m_1}Z_\infty^{m_2}$ with $m_1$, $m_2\in \{-N, 0, +N\}$. As a consequence, $\Tr_\lambda^\omega\bigl([L_0^{T_{N}}] \bigr) $ is again central in $\TT$, which concludes the proof as before.
\end{proof}

We now address Lemma~\ref{lem:ChebSkeinRelationPuncTorus}. 
\begin{lem}
\label{lem:ChebSkeinRelationPuncTorus2}
Suppose that $A^4$ is a primitive $N$--root of unity. 
In the once-punctured torus $T$, let $L_0$, $L_\infty$, $L_1$ and $L_{-1}$ be the curves represented in Figure~{\upshape\ref{fig:PuncturedTorus2}} (or Figure~{\upshape\ref{fig:PuncturedTorus}}). Considering these curves as knots in $T\times [0,1]$ and endowing them with the vertical framing,
$$
[L_0^{T_N}] [L_\infty^{T_N}] = A^{-N^2} [L_1^{T_N}]+ A^{N^2} [L_{-1}^{T_N}]
$$
in the skein algebra $\mathcal S^A(T)$. 
\end{lem}

\begin{proof}
As usual, set $\epsilon=A^{N^2}$ and $\iota = \omega^{N^2}$. 
For the ideal triangulation $\lambda$ indicated in Figure~\ref{fig:PuncturedTorus2}, we can apply  Proposition~\ref{prop:QTraceChebyshev1} and Corollary~\ref{cor:QTraceChebFrobSpecial} to the simple skeins $[L_0]$, $[L_\infty]$, $[L_1]$, $[L_{-1}] \in \mathcal S^A(T)$. Indeed, the edge $e_\infty$ is crossed exactly once by the projections of $L_0$, $L_1$ and $L_{-1}$ to $T$, while $L_\infty$ crosses the edge $e_0$ once. Therefore, 
$ \Tr^\omega _\lambda\bigl ([L_i^{T_N }] \bigr) =\mathbf F^\omega \Bigl( \Tr^\iota _\lambda\bigl ([L_i] \bigr)\Bigr) $ for each $i=0$, $\infty$, $1$, $-1$. 

The skeins $[L_0]$, $[L_\infty]$, $[L_1]$, $[L_{-1}] \in \mathcal S^\epsilon(T)$ satisfy the relation
$$
[L_0] [L_\infty] = \epsilon^{-1} [L_1]+ \epsilon [L_{-1}]. 
$$
Applying the algebra homomorphisms $\Tr^\iota _\lambda \colon \mathcal S^\epsilon(T)  \to \TTT$ and  $\mathbf F^\omega \colon \TTT \to \TT$ on both sides of this equation, and using the property that $ \Tr^\omega _\lambda\bigl ([L_i^{T_N }] \bigr) =\mathbf F^\omega \circ \Tr^\iota _\lambda\bigl ([L_i] \bigr) $, it follows that 
$$
\Tr_\lambda^\omega\bigl([L_0^{T_N}]\bigr)\, \Tr_\lambda^\omega\bigl([L_\infty^{T_N}]\bigr) = \epsilon^{-1}\,  \Tr_\lambda^\omega\bigl([L_1^{T_N}]\bigr) + \epsilon  \,\Tr_\lambda^\omega\bigl([L_{-1}^{T_N}]\bigr) .
$$

The result then follows from the injectivity \cite[Prop.~29]{BonWon1} of the quantum trace homomorphism $\Tr_\lambda^\omega \colon  \mathcal S^A(T) \to \TT$, and from the fact that $\epsilon = A^{N^2}$. 
\end{proof}

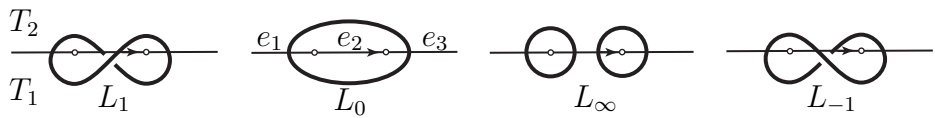
\begin{figure}[htbp]

\SetLabels
( .12*-.3 ) $L_1 $ \\
(.37 * -.35) $ L_0$ \\
( .63* -.3) $L_\infty $ \\
( .88 * -.3) $L_{-1}$\\
( .285*.59 ) $e_1 $ \\
( .37*.59 ) $e_2 $ \\
( .46*.59 ) $e_3 $ \\
( .025*-.2 ) $T_1 $ \\
( .025*.8 ) $T_2 $ \\
\endSetLabels
\centerline{\AffixLabels{ \includegraphics{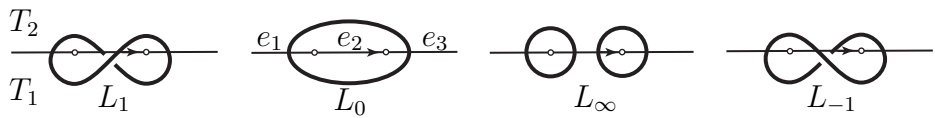} }}
\vskip 10pt

\caption{The $1$--submanifolds $L_1$, $L_0$ and  $L_\infty$  in the twice-punctured plane, with an ideal triangulation $\lambda$}
\label{fig:PuncturedSphere2}
\end{figure}

The proof of Lemma~\ref{lem:ChebSkeinRelationPuncSphere} is very similar. 
\begin{lem}
\label{lem:ChebSkeinRelationPuncSphere2}
Suppose that $A^2$ is a primitive $N$--root of unity with $N$ odd. 
In the twice-punctured plane $U$, let $L_1$, $L_{-1}$, $L_0$ and $L_\infty$  be the $1$--submanifolds represented in Figure~{\upshape\ref{fig:PuncturedSphere2}} (or Figure~{\upshape\ref{fig:PuncturedSphere}}). Considering these submanifolds as links in $U\times [0,1]$ and endowing them with the vertical framing,
\begin{align*}
[L_1^{T_N}]&=  A^{-N^2} [L_0^{T_N}]+ A^{N^2}  [L_\infty^{T_N}]\\
\text{and } [L_{-1}^{T_N}]&=  A^{N^2} [L_0^{T_N}]+  A^{-N^2} [L_\infty^{T_N}]
\end{align*}
in the skein algebra $\mathcal S^A(U)$. 
\end{lem}
\begin{proof} Just observe that, for the ideal triangulation $\lambda$ represented in Figure~\ref{fig:PuncturedSphere2}, we can apply Proposition~\ref{prop:QTraceChebyshev1} and Corollary~\ref{cor:QTraceChebFrobSpecial} to the simple skeins represented by $L_1$, $L_0$, $L_{-1}$, as well as  each of the two components of $L_\infty$. The proof is then identical to that of Lemma~\ref{lem:ChebSkeinRelationPuncTorus2}. 
\end{proof}

This takes care of the proofs that we had postponed up to this point, and in particular completes the proofs of Theorems~\ref{thm:ChebyshevCentral} and \ref{thm:ChebSkeinRelation}. \qed

\subsection{Proof of Theorem~\ref{thm:ChebyQTracesFrob}}
\label{subsect:ProofChebyQTracesFrob}
Assuming that $A^4$ is a primitive $N$--root of unity, we want to prove that the diagram
$$
\xymatrix{
\SSS
 \ar[r]^{\Tr_{\lambda}^\omega} 
 & \TT\\
\SSSS
 \ar[r]^{\Tr_{\lambda}^\iota}
 \ar[u]^{\mathbf T^A}
 & \TTT
 \ar[u]_{\mathbf F^\omega}}
$$
is commutative. Namely, we need to show that $ \Tr^\omega _\lambda\bigl ([K^{T_N }] \bigr) =\mathbf F^\omega \bigl( \Tr^\iota _\lambda\bigl ([K] \bigr)\bigr) $ for every skein $[K] \in \SSSS$. 

Proposition~\ref{prop:QTraceChebyshev1} and Corollary~\ref{cor:QTraceChebFrobSpecial}  prove this property for a specific type of skeins. To extend it to all of $\SSSS$, it suffices to use the following property. Recall that a skein $[K] \in \SSS$ is \emph{simple} if it is represented by a framed knot $K \subset S \times [0,1]$ that projects to a simple closed curve in $S$ and is endowed with the vertical framing.

\begin{lem}
\label{lem:SimpleSkeinCuttingEdgeOnce}
Let $\lambda$ be an ideal triangulation of the surface $S$. The skein algebra $\SSS$ is generated by simple skeins $[K]$ that each cross some edge $e_i$ of $\lambda$ exactly once. 
\end{lem}
\begin{proof} This is an immediate consequence of arguments of Bullock in  \cite{Bull}. 

The consideration of a maximal tree in the dual graph of $\lambda$ provides a set $\{ e_{i_1}, e_{i_2}, \dots, e_{i_u}\}$ of edges of  $\lambda$ that split $S$ into a disk. By considering a handle decomposition of $S$ where the $e_{i_k}$ are the co-cores of the handles, Bullock  proves in \cite[Lemma~3]{Bull} that $\SSS$ is generated by simple skeins that cut each of these edges $e_{i_k}$ in 0 or 1 point.  Because the edges $e_{i_k}$ split the surface $S$ into a disk, we can eliminate from the list of these generators of $\SSS$ those which are disjoint from the $e_{i_k}$, since such simple skeins are scalar multiples of the identity $[\varnothing]$. Each of the remaining generators is a simple skein cutting  at least one edge  $e_{i_k}$ in a single point. 
\end{proof}

For each of the generators $[K]$ of $\SSS$ provided by Lemma~\ref{lem:SimpleSkeinCuttingEdgeOnce}, Proposition~\ref{prop:QTraceChebyshev1} and Corollary~\ref{cor:QTraceChebFrobSpecial} show that $ \Tr^\omega _\lambda\circ \mathbf T^A \bigl ([K] \bigr) =\mathbf F^\omega\circ \Tr^\iota _\lambda\bigl ([K] \bigr) $. Since all maps involved are algebra homomorphisms (using for $\mathbf T^A$ the fact that we just completed the proof of Theorem~\ref{thm:ChebSkeinRelation}), it follows that $ \Tr^\omega _\lambda\circ \mathbf T^A =\mathbf F^\omega\circ \Tr^\iota _\lambda$ over all of $\SSSS$. 
 \qed

\end{document}